\documentclass[default]{elsarticle}

\usepackage{hyperref}

\usepackage{amssymb,amsthm,amsmath}
\usepackage{graphicx}
\usepackage[cmtip,all]{xy}

\newtheorem{theorem}{Theorem}[section]
\newtheorem{lemma}[theorem]{Lemma}

\theoremstyle{definition}

\theoremstyle{remark}

\newtheorem{proposition}[theorem]{Proposition}

\numberwithin{equation}{section}

\usepackage[makeroom]{cancel}
\usepackage{mathtools}
\usepackage{amscd,subeqnarray}
\usepackage{array}
\usepackage{wrapfig}
\usepackage{color}
\usepackage{ragged2e}
\usepackage{subfig}
\usepackage{soul}
\usepackage{boldfonts, bm}
\usepackage{xcolor,cancel}
\usepackage[export]{adjustbox}
\usepackage{tikz, xcolor}
\usetikzlibrary{shapes.geometric}
\usetikzlibrary{shapes,arrows}
\usepackage{siunitx} 
\usepackage{geometry}
\usepackage{verbatim}
\usepackage{slashbox,nicefrac}

\usepackage{pgf}
\usepgflibrary{plothandlers,plotmarks}

\DeclareMathAlphabet\mathbfcal{OMS}{cmsy}{b}{n}

\newcommand{\vertiii}[1]{{\left\vert\kern-0.25ex\left\vert\kern-0.25ex\left\vert #1 
		\right\vert\kern-0.25ex\right\vert\kern-0.25ex\right\vert}}

\newcommand{\grad}{\nabla}
\newcommand{\divv}{\nabla \cdot}

\newcommand{\norm}[1]{\| #1\|}
\newcommand{\enorm}[1]{\| #1\|_{\mathcal{E}}}

\newcommand{\snorm}[1]{| #1 |}
\newcommand{\triplenorm}[1]{%
	\left\vert\kern-0.9pt\left\vert\kern-0.9pt\left\vert #1
	\right\vert\kern-0.9pt\right\vert\kern-0.9pt\right\vert} 
\newcommand{\av}[1]{\left\{ #1 \right\}} 
\newcommand{\jump}[1]{\left\lbrack\!\left\lbrack #1 \right\rbrack\!\right\rbrack} 

\newcommand{\cV}{\mathcal{V}}
\newcommand{\cW}{\mathcal{W}}

\newcommand{\Done}{\mathcal{DG}_1}

\newcommand{\Ch}{\mathcal{CG}_1}
\newcommand{\cVh}{\mathcal{V}_h}

\newcommand{\cWh}{\cW_h}
\newcommand{\bff}{\mathbf{f}}
\newcommand{\bXh}{\bm{X}_h}

\newcommand{\pOm}{\partial \Omega}

\newcommand{\Gamman}{\Gamma_N}
\newcommand{\Gammad}{\Gamma_D}
\newcommand{\Eh}{{{\mathcal E}_h}} 
\newcommand{\Eho}{{{\mathcal E}^{I}_h}} 
\newcommand{\Ehb}{\mathcal{E}_h^{\partial}}
\newcommand{\Ehbd}{\mathcal{E}_h^{\partial, D}}

\newcommand{\Ehbn}{\mathcal{E}_h^{\partial, N}}
\newcommand{\Th}{\mathcal{T}_h}
\newcommand{\K}{T}
\newcommand{\sumK}{\sum_{\K \in \Th}} 
\newcommand{\sumE}{\sum_{e \in \Eh}} 
\newcommand{\sumEo}{\sum_{e \in \Eho}} 
\newcommand{\sumEod}{\sum_{e \in \Eho \cup \Ehbd }} 
\newcommand{\sumEd}{\sum_{e \in \Ehbd}} 
\newcommand{\sumEn}{\sum_{e \in \Ehbn }} 

\newcommand{\Ltwo}{L^2(\Omega)}
\newcommand{\Ltwod}{[L^2(\Omega)]^d}
\newcommand{\Ltwoz}{L^2_0(\Omega)}

\newcommand{\Hone}{H^{1}(\Omega)}

\newcommand{\Honed}{{[H^{1}(\Omega)]^d}}

\newcommand{\Honedzd}{[H^{1}_{0,D}(\Omega)]^d}
\newcommand{\Htwod}{[H^{2}(\Omega)]^d}

\newcommand{\bzero}{\mathbf{0}}

\newcommand{\taub}{{\boldsymbol {\tau}}}
\newcommand{\etau}{\mathbb{\eta}_\bu}
\newcommand{\etap}{\eta_p}
\newcommand{\xiu}{\mathbb{\xi}_\bu}
\newcommand{\xip}{\xi_p}

\newcommand{\bah}{\ba_\theta}
\newcommand{\bbh}{\bb}

\newcommand{\Pih}{\Pi_h}

\newcommand{\buh}{\bu_h}
\newcommand{\ph}{p_h}
\newcommand{\bne}{\bn_e}
\newcommand{\bnK}{\bn_\K}
\newcommand{\nK}{\bnK}
\newcommand{\nne}{\bne}

\newcommand{\half}{\frac{1}{2}}

\newcommand{\dyle}{\displaystyle}

\journal{Computers \& Mathematics with Applications}

\begin{document}

\begin{frontmatter}

\title{An Enriched Galerkin Method for the Stokes Equations}

\author[Yi]{Son-Young Yi}
\ead{syi@utep.edu}

\author[Hu]{Xiaozhe Hu}
\ead{xiaozhe.hu@tufts.edu}

\author[Lee]{Sanghyun Lee\corref{mycorrespondingauthor}}
\cortext[mycorrespondingauthor]{Corresponding author}
\ead{lee@math.fsu.edu}

\author[Adler]{James H. Adler}
\ead{James.Adler@tufts.edu}

\address[Yi]{Department of Mathematical Sciences, University of Texas at El Paso, El Paso, TX 79968}
\address[Hu]{Department of Mathematics, Tufts University, Medford, MA 02155}
\address[Lee]{Department of Mathematics, Florida State University, Tallahassee, FL 32306}
\address[Adler]{Department of Mathematics, Tufts University, Medford, MA 02155}





\begin{abstract}
We present a new enriched Galerkin (EG) scheme for the Stokes equations based on piecewise linear elements for the velocity unknowns and piecewise constant elements for the pressure. The proposed EG method augments the conforming piecewise linear space for velocity by adding an additional degree of freedom which corresponds to one discontinuous linear basis function per element. Thus, the total number of degrees of freedom is significantly reduced in comparison with standard conforming, non-conforming, and discontinuous Galerkin schemes for the Stokes equation. We show the well-posedness of the new EG approach and prove that the scheme converges optimally. For the solution of the resulting large-scale indefinite linear systems we propose robust block preconditioners, yielding scalable results independent of the discretization and physical parameters.  Numerical results confirm the convergence rates of the discretization and also the robustness of the linear solvers for a variety of test problems. 
\end{abstract}


\end{frontmatter}


\section{Introduction}\label{sec:intro}
We consider the Stokes equations in a bounded polyhedral domain, $\Omega\subset\mathbb{R}^d$ ($d=2,3$), with Lipschitz boundary $\Gamma=\partial\Omega$. 
Let $\bu: \Omega\rightarrow\mathbb{R}^d$ be the velocity field of a fluid occupying $\Omega$ and $p:\Omega\rightarrow\mathbb{R}$ denote the kinematic pressure. Given  $\bm{f}\in [L^2(\Omega)]^d$, $\bm{g}\in [H^\half(\Gammad)]^d$, and $\bm{s}\in [H^{-\half}(\Gamman)]^d$, 
the Stokes system for the velocity and the pressure of an incompressible viscous fluid within $\Omega$ is
\begin{subequations}\label{eqn: stokes}
	\begin{alignat}{2}
		- \divv (2 \mu \epsilon (\bu)  - p \bI) & = \bff && \mbox{  in } \Omega,  \label{stokes1} \\
		\divv \bu & = 0 && \mbox{  in } \Omega, \label{stokes2}  \\
		\bu & = \bg && \mbox{ on } \Gammad, \label{bdry-d} \\
		(2 \mu \epsilon (\bu)  - p \bI) \bn & = \bs && \mbox{ on } \Gamman. \label{bdry-t}
	\end{alignat}
\end{subequations}
Here, $\bepsilon(\bu) = \half ( \grad \bu + (\grad \bu)^T )$ is the symmetric part of the gradient of $\bu$, $\bI$ is the identity tensor on $\mathbb{R}^d$, and $\mu > 0$ is the fluid viscosity.
We impose a Dirichlet boundary condition on $\Gammad$, {where $|\Gammad| > 0$,} and a Neumann boundary condition on $\Gamman = \pOm \setminus \Gammad$ with $\bn$ as the unit outward normal vector on $\Gamman$. 
{If $\Gamman = \emptyset$, the following condition is enforced:
\begin{equation}\label{eqn: compat}
\int_{\partial \Omega} \bg \cdot \bn \, ds = 0.
\end{equation}
 }

It is well-known that when solving the Stokes equations, finite-element spaces for the velocity and pressure have to satisfy an {\it inf-sup} stability condition (LBB condition)~\cite{1969Ladyzhenskaya-a,1973Babuska-a,1974Brezzi-a} in order to provide a stable and convergent solution. Roughly speaking, this inf-sup condition requires the proper balance between the velocity and pressure spaces; if the velocity space is not sufficiently large compared to the pressure space, the pressure approximation presents spurious oscillations.  For a detailed study and comprehensive review of inf-sup stable finite elements for the Stokes system, we refer to~\cite{1986GiraultRaviart-a,1991BrezziFortin-a,2013BoffiBrezziFortin-a} {and research cited therein}. 

Another difficulty that arises when designing finite-element methods for the Stokes equations stems from the numerical treatment of the incompressibility constraint (mass conservation); its importance is emphasized in several works (see, e.g.~\cite{2009Linke-a}, \cite{2018LinkeMerdonNeilanNeumann-a} and a review article~\cite{2017VolkerLinkeMerdonNeilanRebholz-a}). Indeed, one can construct conforming finite elements that satisfy the incompressibility constraint point-wise. However, this leads us to a velocity space consisting of high-order or complex elements with limited applicability. One example in two-dimensions (2D) is  the Scott-Vogelius element ($\mathbb{P}_{k}$-$\mathbb{P}_{k-1}, \, k \ge 4$) \cite{1985ScottVogelius-a}. In three dimensions (3D), the situation is more subtle, and the error and stability analysis of the Scott-Vogelius element is yet to be completed.  On the other hand, one can resort to conforming finite-element methods that use lower-order elements but satisfy the incompressibility condition only weakly. Some classical examples of such finite elements include Taylor-Hood, Bernardi-Raugel, and MINI elements. An approach based on analogues of the Bernardi-Raugel element, which satisfy the incompressibility constraint pointwise and use rational basis functions, was proposed and studied in~\cite{2014GuzmanNeilan-b,2014GuzmanNeilan-a}.

A primary goal of our study is to develop inf-sup stable finite-element spaces with minimal number of degrees of freedom for the Stokes system  with mixed boundary conditions on simplicial meshes.  { In addition, we aim to equip this solution framework with an optimal linear solver}. For the pressure variable, we choose the piecewise constant space as it is the lowest-order element. Then, a continuous linear space would seem like an attractive choice for the velocity. 
However, the continuous linear velocity and piecewise constant pressure ($\mathbb{P}_1$-$\mathbb{P}_0$) pair  has long been known to be unstable. In order to make the velocity space sufficiently large to pair with the piecewise constant pressure space, Fortin~\cite{1972Fortin} used continuous quadratic elements for the velocity in 2D. Later, Bernardi and Raugel~\cite{1985BernardiRaugel-a} proposed a slightly smaller velocity space by enriching the continuous linear space with quadratic edge bubble functions on each edge of the mesh elements in 2D. The straightforward extension of Bernardi-Raugel element in 3D is a subspace of the continuous cubic space. To reduce the number of degrees of freedom even further, another approach is to utilize nonconforming elements for the velocity. For example,  Crouzeix and Raviart \cite{1973CrouzeixRaviart} showed that the nonconforming linear space with continuity at the midpoints of element edges for the velocity and the piecewise constant space for pressure  provide a stable pair for the Stokes system with Dirichlet boundary conditions ($\partial \Omega = \Gammad$). However, this approach does not work for the Stokes system if $\Gamman \ne \emptyset$. Kouhia and Stenberg~\cite{1995KouhiaStenberg} circumvented this problem by using a velocity space consisting of conforming linear elements for one component  and nonconforming linear elements for the other component in 2D, producing the lowest-order stable finite-element method for the two-dimensional Stokes system with mixed boundary conditions to date. However, a straightforward generalization of Kouhia and Stenberg's approach to three dimensions does not yield a stable pair with a discrete Korn's inequality, as shown by Hu and Schedensack~\cite{2019HuSchedensack}.  
More recently, in~\cite{LiRui2021a}, the authors used a continuous, piecewise-linear space enriched by the Raviart-Thomas element for the velocity, and piecewise constants for the pressure. It has been shown that such discretization is well-posed and achieves optimal convergence order. However, this usually leads to a large linear system due to the extra degrees of freedom of the Raviart-Thomas element.


In this present work, we consider an Enriched Galerkin (EG) method based on simplicial grids to construct inf-sup stable finite elements for the Stokes system with mixed boundary conditions in both two and three dimensions. EG methods are a special type of finite-element method, whose solution spaces consist of continuous Lagrange finite elements enriched with some discontinuous functions. To compensate the inconsistency of the bilinear forms, EG methods utilize DG-like variational formulations. 
{ 
Thus, EG methods can be easily implemented by modifying standard FEM codes.}
Most existing EG methods enrich linear Lagrange elements with piecewise constant functions \cite{SunLiu09, LeeLeeWhe2016, LeeWhe2016_EG, ChaabaneEtAl2018, ChooLee2018, VamarajuElAl18, KuzHajRupp20, OuardghiEtAl21}.  They have been employed to study various problems, such as modeling flow and transport in porous media \cite{LeeWhe2016_EG, LeeWhe2017_egtwo, rupp2020continuous, kadeethum2020flow, ArbogastTao2019},  the shallow water equations \cite{hauck2020enriched}, and poroelasticity problems \cite{ChooLee2018, GiraultEtAl2020,kadeethum2020finite,kadeethum2020enriched}, and, in particular, the Stokes equations \cite{ChaabaneEtAl2018}.

{
Here, we propose a new EG method, in which the velocity space is the linear Lagrange elements enriched with certain piecewise linear, mean-zero vector functions and whose weak form is the standard weak form used in the interior penalty DG method for the Stokes problem. On the other hand, the EG method developed in \cite{ChaabaneEtAl2018} uses the linear Lagrange elements enriched with piecewise constants for the velocity and employs a non-standard weak form. Our velocity space has nontrivial divergence-free functions and provides inf-sup stable elements when paired with the piecewise constant pressure space for the Stokes equations. 
Indeed, this EG space was also utilized  to  provide locking-free displacement solutions for linear elasticity~\cite{YiLeeZikatanov21}.
With our new approach, we substantially reduce the number of degrees of freedom compared to existing inf-sup stable finite elements for the Stokes equations. In this paper, we prove the LBB stability and optimal-order error estimates for our new EG method. 
}

Finally, we develop and analyze a class of optimal preconditioners for the resulting discrete linear system. There have been several works recently that provide efficient solvers for the Stokes system.  These include multigrid methods on structured or semistructured grids based on distributive relaxation~\cite{2011BacutaVassilevskiZhang-a}, Uzawa, Vanka, and Braess-Sarazin relaxation~\cite{2018LuoRodrigoGasparOosterlee-a,2014GasparRodrigoHeidenreich-a,2017AdlerJ_BensonT_MacLachlanS-aa}, and auxiliary space preconditioning~\cite{2016HongKrausXuZikatanov-a}.  Such multilevel preconditioners require a sequence of grids and a corresponding hierarchy of spaces.  To avoid issues related to systems discretized on fully unstructured grids, we follow an operator preconditioning framework described in \cite{2004LonghinWathen-a,2011MardalWinther-a}.  We obtain similar results as those reported in~\cite{2016Ma-a}, leading to uniform block diagonal, lower, and upper triangular preconditioners.  Moreover, the action of each of these preconditioners is readily computed using algebraic multigrid methods as they only require the solution of symmetric positive definite systems.

The presentation of our results traces the following virtual landscape. The notation and some definitions are reviewed in Section~\ref{sec:prelim}. The EG scheme is described in Section~\ref{sec:EG} and the well-posedness of the discrete problem is shown in Section~\ref{sec:well-posedness}. The error analysis follows in Section~\ref{sec:convergence}, with an introduction and analysis of the block preconditioners in Section~\ref{sec:block-preconditioners}. Numerical results are reported in Section~\ref{sec:numerical-results}, and concluding remarks and a discussion of future work are given in Section~\ref{sec:conclusions}.


\section{Notation and preliminaries}\label{sec:prelim}
We introduce the notation and preliminary results needed throughout the rest of the paper.  
Following \cite{adams2003sobolev}, we use the standard notation for Sobolev spaces $H^{s}(E)$ for a domain $E\subset\mathbb{R}^d$, and integer $s$, corresponding to spaces of functions with square integrable derivatives up to order $s$. We also define the space 
$\Ltwoz$ of functions from $\Ltwo$ with zero mean. The Sobolev norm and seminorm  associated with $H^s(E)$ are denoted by $\|\cdot\|_{s,E}$ and $|\cdot|_{s,E}$, respectively. These definitions extend in an obvious fashion to vector and tensor-valued functions. 
We note the standard convention $H^0(E)=L^2(E)$. When we would like to emphasize the role of the domain $E$, we denote by $(\cdot, \cdot)_E$ the $L^2(E)$-inner product on $E$.

We consider a shape regular, simplicial partition of the computational domain,
$\displaystyle \overline{\Omega}=\cup_{\K\in \Th}\overline{\K}$ where
$\K\in\Th$ are triangles when $d=2$ and tetrahedrons when $d=3$. The
so called ``broken'' Sobolev spaces $H^{s}(\Th)$ are the natural
spaces to work with when analyzing DG and EG methods: these are the spaces of square integrable functions on
$\Omega$ whose restrictions on $\K\in\Th$ are in $H^{s}(\K)$. Next, we denote by $\nK$
the unit outward normal vector to $\partial \K$, $\K \in \Th$ and introduce the characteristic mesh size
\[
h=\max_{\K\in\Th} h_{\K}, \quad h_{\K}=\operatorname{diam}(\K),\;\;\K \in \Th.
\]
The collection of all interior edges or faces is denoted by $\Eho$, and $\Ehb$ represents the collection of all boundary edges or faces.
Then, $\Eh = \Eho \cup \Ehb$ is the collection of all edges or faces in $\Th$. Moreover, $\Ehb$ is decomposed into two subsets: 
$\Ehbd$ and $\Ehbn$, which are the collections of the boundary edges or faces belonging to $\Gammad$ and $\Gamman$, respectively. 
For any $e \in \Eho$, there are  two neighboring elements $\K^{+}$ and $\K^{-}$  such that  $e = \partial \K^{+}\cap \partial \K^{-}$. We associate one unit normal vector $\nne$ with $e$, which is assumed to be  oriented from $\K^+$ to $\K^-$. If $e \in \Ehb$, then $\nne$ is taken to be the unit outward normal vector  to $\partial \Omega$.
For $e \in \Eho$, in general, there will be two traces of a function $\zeta$ on $e$, i.e., $\zeta$ is not single-valued on $e$.
Thus, for $\zeta\in H^{1}(\Th)$ and $e=\partial \K^{+}\cap \partial \K^{-}$, $\zeta^{\pm}$ are the traces of $\zeta$ on $e$.
We then define the average and jump operators,  $\av{\cdot}$ and $\jump{\cdot}$ as follows:
\begin{equation*}
	\begin{aligned}
		& \av{\zeta} := \dfrac{1}{2}\left( \zeta^+ +  \zeta^- \right), \quad \jump{\zeta} = \zeta^+- \zeta^-, \quad \mbox{for}\quad e\in \Eho.\\
		& \av{\zeta} = \jump{\zeta} :=  \zeta, \quad \mbox{for}\quad e \in \Ehb.
	\end{aligned}
\end{equation*}

Finally, let $|\K|$ denote the volume of a simplex $\K$ in $\mathbb{R}^d$, let $|e|$ denote the area/length of face $e\in\partial\K$, and define
$\mathbb{P}_k(\K)$ to be the space of polynomials of total degree at most $k\ge 0$.
In the analysis that follows, we also use the well-known trace inequalities which hold for any $v \in H^s(\K), \, s \geq 1$, $q\in \mathbb{P}_k(\K)$, and for every $e \subset \partial \K$ and $\K\in \Th$,
\begin{eqnarray}
	\norm{v}_{0, e} &\leq& C_t |e|^{1/2} | \K |^{-1/2} \left( \norm{v}_{0, \K} + h_\K \norm{\nabla v}_{0, \K}\right),
	\label{eqn:trace1}\\
	\norm{q}_{0, e} &\leq& \widetilde{C}_t(k) |e|^{1/2} | \K |^{-1/2} \norm{q}_{0, \K}.
	\label{eqn:trace2}
\end{eqnarray}
Here, $C_t$ is independent of $h_\K$ and $\widetilde{C}_t(k)$ depends on the polynomial degree $k$.
We note that the definitions of $\jump{\cdot}$ and $\av{\cdot}$ given earlier as
well as the trace inequalities above are easily shown to hold also for  vector and tensor-valued functions.

\section{Weak Form and Enriched Galerkin Scheme}\label{sec:EG}
In this section, we present a weak form of the model problem \eqref{eqn: stokes} and propose an enriched Galerkin method for discretization. 
If $| \Gamman | = 0$,  the pressure $p$ is uniquely defined up to an additive constant. Therefore,  we assume that $p$ is mean-zero  in that case. Then, a weak solution to the problem \eqref{eqn: stokes} is the pair $(\bu, p) \in \Honed \times \Ltwoz$ if $| \Gamman | = 0$ or $(\bu, p) \in \Honed \times \Ltwo$ if $| \Gamman | > 0$ such that $\bu|_{\Gammad} = \bg$ and
\begin{subequations}\label{eqn: weak}
\begin{alignat}{2}
2 \mu (\epsilon(\bu), \epsilon (\bv)) - (p, \divv \bv) & = (\bff, \bv) + (\bs, \bv)_{\Gamman},  \quad  &&\forall \bv \in \Honedzd, \label{weak1} \\
-(\divv \bu, w ) & = 0, \quad && \forall w \in \Ltwoz \text{ if } |\Gamman| =0 \text { or }  \Ltwo \text{ if } |\Gamman| > 0. \label{weak2}
\end{alignat}
\end{subequations}
Here, $\Honedzd$ is the space of vector valued functions from $\Honed$ whose traces on $\Gammad$ vanish.

To introduce the EG finite-element space for the velocity variables, we start by considering the vector-valued linear CG finite-element space:
\begin{equation*}
	\Ch := \left \{ \psi \in \Honed ~| \ \psi|_{\K} \in [\mathbb{P}_1(\K)]^d, \ \forall \K \in \mathcal{T}_h \right \}.
\end{equation*}
Then, the EG space for velocity, defined as $\cVh$, is obtained by extending $\Ch$ with the space,
\[
\Done = \left \{ \psi \in \Ltwod \ | \ \psi|_{\K} = c_\K (\bx - \bx_\K), \, c_\K \in \mathbb{R}, \ \forall \K \in \mathcal{T}_h \right \},
\]
where $\bx = [ x_1, \cdots, x_d]^T$ and $\bx_\K$ is the centroid of $\K \in \Th$.
Since $\Ch$ and $\Done$ are disjoint spaces, we have that
\[
\cVh :=   \Ch \oplus \Done \subset \Ltwod.
\]
For the pressure $p$, we simply use the piecewise constant function space: 
\[
\cWh : = \begin{cases}
 \{ \phi \in \Ltwo \ | \ \phi|_\K \in \mathbb{P}_0(\K), \ \forall \K \in \Th\} \cap \Ltwoz & \text{if } | \Gamman | = 0, \\
 \{ \phi \in \Ltwo \ | \ \phi|_\K \in \mathbb{P}_0(\K), \ \forall \K \in \Th\} & \text{if } | \Gamman | > 0.
 \end{cases}
\]
With these spaces, we define the following bilinear and linear forms to discretize \eqref{eqn: weak}:
\begin{align*}
	\bah(\bu, \bv) & := 
	2 \mu \sum_{\K \in \Th} (\epsilon(\bu), \epsilon (\bv) )_\K
	- 2 \mu \dyle\sumEod ( \av{ \epsilon(\bu)}\bne, \jump{\bv} )_e \nonumber \\
	& \qquad + 2 \mu \theta  \dyle \sumEod  ( \jump{\bu}, \av{ \epsilon (\bv)} \bne)_e 
	+  2 \mu \dyle\sumEod  \frac{\alpha}{h_e} ( \jump{\bu}, \jump{\bv} )_e, \nonumber  \\
	\bbh(\bv, w) & := - \sumK(w, \divv \bv)_\K + \sumEod (\av{w}, \jump{\bv}\cdot \bne)_e, \\
	\bF_\theta(\bv) & = \sumK(\bff, \bv)_\K + \sumEn (\bs, \bv)_{e}   +  
	2 \mu \theta  \dyle \sumEd  ( \jump{\bg}, \av{ \epsilon (\bv)} \bne)_e +  2 \mu \dyle\sumEd  \frac{\alpha}{h_e}  ( \jump{\bg}, \jump{\bv} )_e,
\end{align*}
where $h_{e} = |e|^{\frac{1}{d-1}}$, $\theta$ is a symmetrization parameter chosen from $\{-1, 0, 1 \}$,  and $\alpha >0$ is called the penalty parameter. In general, $\alpha$ may vary over each edge in the domain, but for simplicity we assume that it is constant in this paper. 
With the above definitions, the EG method for solving the Stokes equations reads: Find $(\buh, \ph) \in \cVh \times \cWh$ such that
\begin{subequations}\label{eqn: eg}
	\begin{alignat}{2}
		\bah(\buh, \bv) + \bbh(\bv, \ph) & = \bF_\theta(\bv) , && \quad \forall \bv \in \cVh, \label{eg1} \\
		\bbh(\buh, w) & = \sumEd (\av{w}, \jump{\bg}\cdot \bne)_e, && \quad \forall w \in \cWh. \label{eg2}
	\end{alignat}
\end{subequations}
Here, the choice of $\theta$ yields three different methods: NIPG if $\theta = 1$, SIPG if $\theta = -1$, and IIPG if $\theta = 0$. 

To conclude this section, we prove that the discrete weak form is consistent.  
\begin{lemma}[Consistency]\label{lem: consis}
	Assume that the weak solution, $(\bu, p)$, belongs to $[H^t(\Omega)]^d \times H^s(\Omega)$, $t > 1.5$ and $s > 0.5$, and that the discrete EG problem, \eqref{eqn: eg}, has a solution, $(\bu_h,p_h)$. Then, \eqref{eqn: eg} is consistent with the weak problem, \eqref{eqn: weak}, in the sense that
	\begin{subequations}\label{eqn: consis}
		\begin{alignat}{3}
			\bah(\bu-\buh, \bv) +\bbh(\bv, p - \ph) &  = \bzero,  && \quad  \bv \in  \cVh, \label{consis1} \\
			\bbh(\bu - \buh, w ) &  = 0, &&\quad w \in \cWh.
			\label{consis2}  
		\end{alignat}
	\end{subequations}
\end{lemma}
\begin{proof}
  If we multiply  \eqref{stokes1} by any test function $\bv \in \cVh$, use integration by parts on each mesh element $\K \in \Th$, and then sum over $\K \in \Th$, we have on the left side
  \begin{equation}\label{eqn: consis1}
	\begin{aligned}
		& - \sumK ( \divv (2 \mu \epsilon (\bu)  - p \bI), \bv)_\K  \\
		& = -\sumK ((2 \mu \epsilon (\bu) -p \bI)  \nK, \bv)_{\partial \K} 
		+ \sumK ( 2 \mu \epsilon (\bu)  - p \bI, \nabla \bv)_\K  \\
		& = -\sumE (\av{(2 \mu \epsilon (\bu) -p \bI) \nne}, \jump{\bv})_{e}  -\sumEo (\jump{(2 \mu \epsilon (\bu) -p \bI)  \nne}, \av{\bv})_{e}   \\
		& \qquad + 2 \mu \sumK  (\epsilon(\bu), \epsilon (\bv) )_\K -\sumK (p, \divv \bv)_\K.
              \end{aligned}
              \end{equation}
 Then, using the Neumann boundary condition \eqref{bdry-t}  and noting that the second term on the right is zero when  $\bm{u}$ and $p$ are solutions to~\eqref{stokes1}--\eqref{bdry-t}
        we obtain
        \begin{equation*} 
	\begin{aligned}
		& - \sumK ( \divv (2 \mu \epsilon (\bu)  - p \bI), \bv)_\K  \\
		& = -\sumEn(\bs, \bv)_e - 2 \mu  \sumEod (\av{\epsilon (\bu)},  \jump{\bv})_{e}  + \sumEod ( \av{p} , \jump{\bv}\cdot \nne)_{e}
		\\
		& \qquad + 2 \mu \sumK  (\epsilon(\bu), \epsilon (\bv) )_\K -\sumK (p, \divv \bv)_\K  \\
		& \qquad \quad  +  2 \mu \theta  \dyle \sumEo  ( \jump{\bu}, \av{ \epsilon (\bv)} \bne)_e 
		+ 2 \mu  \dyle\sumEo  \frac{\alpha}{h_e} ( \jump{\bu}, \jump{\bv} )_e,
              \end{aligned}
              \end{equation*}
    where the last two terms were added without changing the above bilinear form since $\jump{\bu} = 0$ on any $e \in \Eho$. On the other hand, we have $ \sumK(\bff, \bv)_\K$ on the right side. 
              Next, by adding
             \(2 \mu \theta  \dyle \sumEd ( \jump{\bu}, \av{ \epsilon (\bv)} \bne)_e 
             +  \dyle 2\mu\sumEd  \frac{\alpha}{h_e} ( \jump{\bu}, \jump{\bv} )_e\)
             to both sides and using the Dirichlet boundary condition on $\Gammad$, we obtain
	\begin{subequations}\label{eqn: consis4}
		\begin{equation}\label{eqn: consis2}
			\bah(\bu, \bv) + \bbh(\bv, p)  = \bF_\theta(\bv) \quad \forall \bv \in \cVh.
		\end{equation}
		Similar calculations yield
		\begin{equation}\label{eqn: consis3}
			\bbh(\bu, w)  = \sumEd (\av{w}, \jump{\bg}\cdot \bne)_e \quad \forall w \in \cWh.
		\end{equation}
	\end{subequations}
	Hence, the consistency equations \eqref{eqn: consis} follow by subtracting \eqref{eqn: eg} from \eqref{eqn: consis4}.
\end{proof}

\section{Well-posedness}\label{sec:well-posedness}
The well-posedness of the proposed EG scheme, \eqref{eqn: eg}, for solving the Stokes equations follows from Brezzi thoery~\cite{1974Brezzi-a} by verifying the appropriate continuity/coercivity conditions on the bilinear forms $\ba_{\theta}$ and $\bbh$.  We also follow the Babuska theory~\cite{1973Babuska-a}, which will be useful later for designing parameter-robust preconditioners.

\subsection{Coercivity and Continuity of $\ba_{\theta}$}
We introduce the following inner product for $\alpha >0$:
\begin{equation*}
  (\bu, \bv)_{\mathcal{E}} := 2\mu \left( \sumK (\epsilon(\bu), \epsilon(\bv))_\K + \sumEod \frac{\alpha}{h_e} (\jump{\bu}, \jump{\bv})_{e} \right),	
\end{equation*}
which induces the energy norm 
\begin{equation*}
	\enorm{\bv} := \sqrt{(\bv,\bv)_{\mathcal{E}}}.
\end{equation*}
We first show the coercivity and continuity of the bilinear $\ba_{\theta}$ on $\cV_h$. Assuming that $| \Gammad | >0$, the generalized Korn's inequality in the broken Sobolev space $H^1(\Th)$ states that there is a constant $C_{K} > 0$ such that
\begin{equation*}
	\sumK \norm{\bv}_{1, \K}^2 \leq C_K \left( \sumK \norm{\epsilon(\bv)}_{0, \K}^2  + \sumEod \frac{1}{h_e} \norm{\jump{\bv}}_{0, e}^2 \right).
\end{equation*}

\begin{lemma}[Coercivity of $\bah$]\label{lem: u-coer}
	Assume that $\alpha$ is sufficiently large for the cases of $\theta = -1$ or $\theta = 0$, and any value for $\theta=1$. Then, there exists a constant $\chi > 0$ independent of $h$ such that 
	\begin{equation}\label{eqn: coer}
		\bah(\bv, \bv) \ge \chi \enorm{\bv}^2, \quad \forall \bv \in \cVh.
	\end{equation}
\end{lemma}
\begin{proof}
	For any $\bv \in \cVh$, 
	\begin{equation}\label{eqn: coer1}
		\bah(\bv, \bv)  = 
		2 \mu \left( \sum_{\K \in \Th} (\epsilon(\bv), \epsilon (\bv) )_\K
		-  (1 - \theta ) \dyle \sumEod  ( \jump{\bv}, \av{ \epsilon (\bv)} \bne)_e 
		+   \dyle\sumEod  \frac{\alpha}{h_e} ( \jump{\bv}, \jump{\bv} )_e\right).
	\end{equation}
	If $\theta = 1$, $\bah(\bv, \bv) $ reduces to
	\[
	\bah(\bv, \bv)  
	=  2 \mu \left( \sum_{\K \in \Th} \norm{\epsilon(\bv)}_{0, \K}^2
	+  \sumEod  \frac{\alpha}{h_e}  \norm{\jump{\bv}}_{e}^2 \right).
	\]
	Therefore, \eqref{eqn: coer} holds true with $\chi = 1$.  On
        the other hand, if $\theta = 0$ or $\theta = -1$, we first
        estimate the second term in the parentheses in \eqref{eqn: coer1}  by
	\begin{align}\label{eqn: coer2}
		& | - (1 -  \theta ) \dyle \sumEod  ( \jump{\bv}, \av{ \epsilon (\bv)} \bne)_e |  \nonumber \\
		&\leq  (1 -  \theta )\left(  \sumEod  \frac{\alpha}{h_e}  \norm{\jump{\bv}}_{e}^2 \right)^\half 
		\left( \sumEod   \frac{h_e}{\alpha}  \norm{ \av{ \epsilon (\bv)} \bne}_e^2 \right)^\half  \nonumber \\
		&\leq  (1 -  \theta )  \frac{\sqrt{C}}{\sqrt{\alpha}}  \left(  \sumEod  \frac{\alpha}{h_e}  \norm{\jump{\bv}}_{e}^2 \right)^\half 
		\left( \sumK   \norm{ \epsilon (\bv)}_{0, \K}^2 \right)^\half  \nonumber \\
		&\leq \frac{1}{2}   \sumK  \norm{ \epsilon (\bv)}_{0, \K}^2 
		+  \frac{1}{2}(1- \theta)^2 \frac{C}{\alpha} \sumEod  \frac{\alpha}{h_e}  \norm{\jump{\bv}}_{e}^2,
	\end{align}
        where $C$ is a generic constant depending on $\widetilde{C}_t$ from~\eqref{eqn:trace2}.
	Then, using \eqref{eqn: coer1} and \eqref{eqn: coer2} together, we obtain
	\begin{align*}
		\bah(\bv, \bv) & = 
		2 \mu \left( \sumK   \norm{ \epsilon (\bv)}_{0, \K}^2
		-  (1 - \theta ) \dyle \sumEod  ( \jump{\bv}, \av{ \epsilon (\bv)} \bne)_e 
		+  \sumEod \frac{\alpha}{h_e} \norm{\jump{\bv}}_{0, e}^2 \right) \\
		& \ge 2\mu  \left(\frac{1}{2} \sum_{\K \in \Th} \norm{\epsilon(\bv)}_{0, \K}^2
		+ \left( 1 - \frac{1}{2}(1 -  \theta )^2  \frac{C}{\alpha} \right)\sumEod \frac{\alpha}{h_e} \norm{\jump{\bv}}_{0, e}^2 \right).
	\end{align*}
	If we take $\alpha \ge (1 -  \theta )^2C$, then the coefficient of the second term is no less than $1/2$.
	As a result, we have
	\begin{align*}
		\bah(\bv, \bv) & \ge  2\mu \left( \frac{1}{2} \sum_{\K \in \Th} \norm{\epsilon(\bv)}_{0, \K}^2 + \half \sumEod \frac{\alpha}{h_e} \norm{\jump{\bv}}_{0, e}^2  \right)
		\ge \half  \enorm{\bv}^2,
	\end{align*}
	from which \eqref{eqn: coer} follows with $\chi = \half$.
\end{proof}

\begin{lemma}[Continuity of $\bah$]\label{lem:u-conti}
	Assume that $\alpha$ is sufficient large, then we have
	\begin{equation} \label{ine:u-conti}
		|\bah(\bu, \bv)| \leq 2 \enorm{\bu} \enorm{\bv}, \quad \forall \, \bu, \ \bv \in \cV_h.
	\end{equation}	
\end{lemma}
\begin{proof}
	For any $\bu, \ \bv \in \cV_h$ and following the similar steps in the proof of Lemma~\ref{lem: u-coer}, we have
	\begin{align*}
          |\bah(\bu, \bv) |
          & = 2\mu \left (\sum_{\K \in \Th} (\epsilon(\bu), \epsilon(\bv))_T - \sumEod (\av{\epsilon(\bv)}\bne. \jump{\bv})_e  \right. \\
          & \left.  \quad + \theta \sumEod(\jump{\bu}, \av{\epsilon(\bv)}\bne)_e   + \sumEod \frac{\alpha}{h_e}(\jump{\bu}, \jump{\bv})_e \right ) \\
		& \leq 2 \mu \left\{  \left(\sumK \norm{\epsilon(\bu)}_{0,T}^2\right)^{\frac{1}{2}}  \left(\sumK \norm{\epsilon(\bv)}_{0,T}^2\right)^{\frac{1}{2}} \right. \\
		& \quad  + \left( \sumEod \frac{h_e}{\alpha} \| \av{\epsilon(\bu)}\bne \|_{0, e}^2 \right)^{\frac{1}{2}} \left( \sumEod \frac{\alpha}{h_e} \norm{\jump{\bv}}_{0, e}^2 \right)^{\frac{1}{2}} \\
		& \quad  + |\theta| \left( \sumEod \frac{\alpha}{h_e} \norm{\jump{\bu}}_{0, e}^2 \right)^{\frac{1}{2}}\left( \sumEod \frac{h_e}{\alpha} \| \av{\epsilon(\bv)}\bne \|_{0, e}^2 \right)^{\frac{1}{2}}  \\
		&  \left. \quad + \left( \sumEod \frac{\alpha}{h_e} \norm{\jump{\bu}}_{0, e}^2 \right)^{\frac{1}{2}}\left( \sumEod \frac{\alpha}{h_e} \norm{\jump{\bv}}_{0, e}^2 \right)^{\frac{1}{2}} \right \}.  \\
	\end{align*}
        Using Cauchy-Schwarz inequality then shows that 
	\begin{align*}
          |\bah(\bu, \bv) |
          & \leq 2 \mu\left ( \sumK \norm{\epsilon(\bu)}_{0,T}^2 + \sumEod \frac{h_e}{\alpha} \| \av{\epsilon(\bu)}\bne \|_{0, e}^2 + \left(1+|\theta|\right) \sumEod \frac{\alpha}{h_e} \norm{\jump{\bu}}_{0, e}^2  \right)^{\half} \cdot \\
          & \quad \quad \   \left ( \sumK \norm{\epsilon(\bv)}_{0,T}^2 + |\theta|\sumEod \frac{h_e}{\alpha} \| \av{\epsilon(\bv)}\bne \|_{0, e}^2 + 2 \sumEod \frac{\alpha}{h_e} \norm{\jump{\bv}}_{0, e}^2  \right )^{\half}.
	\end{align*}	
	The trace inequality, \eqref{eqn:trace2} and the fact that $|\theta| \leq 1$ then lead to
	\begin{align*}
          |\bah(\bu, \bv) |
          & \leq 2 \mu \left\lbrack \left(1 + \frac{C}{\alpha}\right) \sum_{\K \in \Th} \norm{\epsilon(\bu)}_{0, \K}^2  +  2 \sumEod \frac{\alpha}{h_e} \norm{\jump{\bu}}_{0, e}^2 \right\rbrack^{1/2} \cdot \\
          & \quad \quad \ \left\lbrack \left(1 + \frac{C}{\alpha}\right) \sum_{\K \in \Th} \norm{\epsilon(\bv)}_{0, \K}^2  +  2 \sumEod \frac{\alpha}{h_e} \norm{\jump{\bv}}_{0, e}^2 \right\rbrack^{1/2}.
	\end{align*}	
        Then,~\eqref{ine:u-conti} follows directly when we choose $\alpha > C$ and we note that,
        as in the proof of Lemma~\ref{lem: u-coer}, the constant $C$ may depend on
        	$\widetilde{C}_t$ from \eqref{eqn:trace2} but is otherwise independent of the mesh size $h$ and $\mu$.
\end{proof}


\subsection{Discrete Inf-Sup Condition and Continuity of $\bbh$}
Next, we prove the discrete inf-sup condition for the pair of $\cVh$ and $\cWh$.  In order to prove the inf-sup condition, and later the error estimates, we need an interpolation operator from $[H^1(\Omega)]^d$ to the EG finite-element space, $\cV_h$, with a commutativity property. 
Such an interpolation operator, $\Pih$, was introduced in \cite{YiLeeZikatanov21}, and we state some useful properties of $\Pih$ here. 

\begin{proposition}
	There exists an interpolation operator $ \Pih: [H^1(\Omega)]^d \to \cVh$ that satisfies the following properties.
	\begin{subequations}
		\begin{alignat}{2}
			& | \bv -  \Pih  \bv|_j \leq C h^{m-j} | \bv|_m, \, \, 0 \leq j \leq m \leq 2, \quad && \forall  \bv \in [H^2(\Omega)]^d, \label{eqn:Pih_bd} \\
			& \snorm{\divv (  \bv -  \Pih  \bv)}_j \leq C h^{1-j} \snorm{\divv  \bv}_1, \, \, j=0, 1, \quad && \forall  \bv \in [H^2(\Omega)]^d, \label{eqn:div_err} \\
			&(\divv (  \bv -  \Pih  \bv), 1)_T = 0, \quad && \forall \K \in \Th, \label{eqn:div_pres} \\
			& \calP_0 (\divv  \bv) = \divv ( \Pih  \bv), \quad && \forall \K \in \Th, \label{eqn:comm} 
		\end{alignat}
	\end{subequations}
	where $\calP_0: H^1(\Omega)\to \cWh$ is the local $L^2$-projection and satisfies
	\begin{equation*}
		\snorm{ w - \calP_0 w }_j \leq C h^{1-j} \snorm{w}_1, \ \  j=0, 1, \quad \forall w \in H^1(\Omega).
	\end{equation*}
\end{proposition}

\begin{lemma}
	For any $ \bv \in \Honedzd$,
	\begin{subequations}\label{cor:enorm_bd}
		\begin{align}
			& \enorm{ \bv -  \Pih  \bv} \leq \sqrt{2 \mu}  C \snorm{ \bv}_1, \label{eqn:enorm_bd1} \\
			& \enorm{ \Pih  \bv} \leq \sqrt{2 \mu}  C^* \snorm{ \bv}_1. \label{eqn:enorm_bd2}
		\end{align}
	\end{subequations}
\end{lemma}
\begin{proof}
	The first inequality, \eqref{eqn:enorm_bd1}, is a direct consequence of \eqref{eqn:Pih_bd} and \eqref{eqn:trace1}:
	\begin{align*}
		\enorm{ \bv -  \Pih  \bv}^2 & = 2 \mu \left[\sumK \norm{\epsilon( \bv -  \Pih  \bv)}_{0, \K}^2 
		+ \sumEod \frac{\alpha}{h_e} \norm{ \jump{ \bv -  \Pih  \bv}}_{0, e}^2 \right].
	\end{align*}
To prove \eqref{eqn:enorm_bd2},  we first recall that  $\jump{ \bv }= 0$ on every  $e \in \Eho \cup \Ehbd$. Therefore, 
	\[
	\enorm{ \bv} = \sqrt{2 \mu} \norm{\epsilon (\bv)}_0, \quad \forall  \bv \in \Honedzd.
	\]
	 Then,  the bound in~\eqref{eqn:enorm_bd1} and the triangle inequality give the desired result. 
	
\end{proof}
Using the result above, we now prove the inf-sup condition.
\begin{lemma} \label{lem:inf-sup}
	For a sufficiently large $\alpha$, there exists a constant $\beta >0$, independent of $h$ and $\mu$, such that 
	\begin{equation}\label{eqn: infsup}
		\inf_{q \in \cWh}\sup_{\bv \in \cVh} \frac{\bbh(\bv, q)}{\enorm{\bv}\norm{q}_0} \ge \frac{\beta}{\sqrt{2 \mu}}.
	\end{equation}
\end{lemma}
\begin{proof}
	For any $q \in \cWh \subseteq \Ltwo$, there exists a vector $\bv \in \Honedzd$ and a constant $C^{**} > 0$ such that \cite{GopalQiu12,ern2021finite}
	\begin{equation}\label{eqn: divbv}
		\divv \bv = - q \quad \mbox{and} \quad \snorm{\bv}_1 \leq C^{**} \norm{q}_0. 
	\end{equation}
	Then, using \eqref{eqn:div_pres}, \eqref{eqn:enorm_bd2}, \eqref{eqn: divbv}, followed by the trace inequality, we have
	\begin{align*}
		\frac{\bbh(\Pih \bv, q)}{\enorm{\Pih \bv}} & = \frac{- \sumK(q, \divv \Pih \bv)_\K}{\enorm{\Pih \bv}} + \frac{\sumEod (\av{q}, \jump{\Pih \bv}\cdot \bne)_e}{\enorm{\Pih \bv}} \\
		& = \frac{- \sumK(q, \divv  \bv)_\K}{\enorm{\Pih \bv}} + \frac{\sumEod (\av{q} \bne, \jump{\Pih \bv})_e}{\enorm{\Pih \bv}}  \\
		&  \ge \frac{\norm{q}_0^2}{\sqrt{2 \mu} C^* C^{**} \norm{q}_0} - \frac{\sumEod \norm{\av{q} \bne}_{0, e} \norm{ \jump{\Pih \bv}}_{0, e}}{\enorm{\Pih \bv}} \\
		&  \ge \frac{\norm{q}_0^2}{\sqrt{2 \mu} C^*  C^{**}\norm{q}_0} - \frac{1}{\enorm{\Pih \bv}} \left( \frac{1}{2 \mu}  \sumEod \frac{h_e}{\alpha} \norm{\av{q} \bne}_{0, e}^2 \right)^\half \left( 2 \mu  \sumEod  \frac{\alpha}{h_e}  \norm{ \jump{\Pih \bv}}_{0, e}^2\right)^\half \\
		&  \ge \frac{\norm{q}_0^2}{\sqrt{2 \mu} C^*  C^{**}\norm{q}_0} -  \frac{\sqrt{C}}{\sqrt{2 \mu} \sqrt{\alpha}} \frac{\norm{q}_0 \enorm{\Pih \bv}}{\enorm{\Pih \bv}} \\
		& = \frac{1}{\sqrt{2 \mu}} \left( \frac{1}{C^* C^{**}} - \frac{\sqrt{C}}{\sqrt{\alpha}}\right) \norm{q}_0,
	\end{align*}
       where $C$ depends on $\widetilde{C}_t$ from~\eqref{eqn:trace2}.
	If $\alpha$ is sufficiently large, then $( \frac{1}{C^* C^{**}} - \frac{\sqrt{C}}{\sqrt{\alpha}})$ is positive, and the inf-sup condition \eqref{eqn: infsup} holds true with $\beta = ( \frac{1}{C^* C^{**}} - \frac{\sqrt{C}}{\sqrt{\alpha}})$.
\end{proof}

In the next lemma, we show the continuity of the bilinear form $\bbh(\cdot, \cdot)$.
\begin{lemma}\label{lem:b-conti}
	For any $\bv \in \cV_h$ and $q \in \cWh$, for sufficiently large $\alpha$, the bilinear form $\bbh(\cdot, \cdot)$ is continuous:
	\begin{equation*}
		|\bbh(\bv, q)| \leq \frac{\sqrt{d}}{\sqrt{2 \mu}} \|q \|_0 \enorm{\bv}. 
	\end{equation*}
\end{lemma}
\begin{proof}
	By the Cauchy-Schwarz inequality,
	\begin{align*}
          \left|\sum_{\K \in \Th}(q, \divv \bv)_T\right|^2
          & \leq 
          \left(\sum_{\K \in \Th}\|q\|_{0,\K}^2\right)
            \left( \sum_{\K \in \Th} \| \divv \bv \|^2_{0, \K} \right)
            = 
          \|q\|^2_{0}\sum_{\K \in \Th} \| \divv \bv \|^2_{0, \K}.
        \end{align*}
	Using this bound,
	\begin{eqnarray*}
          |\bbh(\bv, q)|
          &=&  \left|- \sum_{\K \in \Th}(q, \divv \bv) + \sumEod (\av{q}, \jump{\bv} \cdot \bne)_e\right| \\
          & \leq & \frac{1}{\sqrt{2\mu}} \| q \|_0 \, \sqrt{2\mu} \left( \sum_{\K \in \Th} \| \divv \bv \|^2_{0, \K} \right)^{\half} \\
          && + \left( \frac{1}{2\mu}\sumEod \frac{h_e}{\alpha} \|\av{q} \bne \|_{0, e}^2 \right)^{\half} \left( 2 \mu \sumEod \frac{\alpha}{h_e} \| \jump{\bv} \|_{0,e}^2 \right)^{\half}.
	\end{eqnarray*}
Then, the trace inequality \eqref{eqn:trace2} leads to 
	\begin{align*}
          |\bbh(\bv, q)|
		& \leq \frac{1}{\sqrt{2\mu}} \| q \|_0 \,  \left( 2\mu \sum_{\K \in \Th} \| \divv \bv \|^2_{0, \K} \right)^{\half} + \frac{1}{\sqrt{2 \mu}}\frac{\sqrt{C}}{\sqrt{\alpha}} \|q \|_0 \left( 2 \mu \sumEod \frac{\alpha}{h_e} \| \jump{\bv} \|_{0,e}^2 \right)^{\half} \\
		& \leq \frac{1}{\sqrt{2\mu}} \| q \|_0 \,  \left( 2\mu \, d \sum_{\K \in \Th} \|\epsilon(\bv)  \|^2_{0, \K} \right)^{\half} + \frac{1}{\sqrt{2 \mu}}\frac{\sqrt{C}}{\sqrt{\alpha}} \|q \|_0 \left( 2 \mu \sumEod \frac{\alpha}{h_e} \| \jump{\bv} \|_{0,e}^2 \right)^{\half} \\
		&  \leq  \frac{\sqrt{d}}{\sqrt{2 \mu}} \| q \|_0 \enorm{\bv}
	\end{align*}
provided $\alpha > C$, which completes the proof. 
\end{proof}

\subsection{Babuska's Theory}
In the previous two subsections, we verified the assumptions on the bilinear forms $\bah$ and $\bbh$ via Brezzi's theory~\cite{1974Brezzi-a}, which naturally implies the well-posedness of the EG scheme~\eqref{eqn: eg}. In order to simplify the convergence analysis as well as develop the parameter-robust preconditioners, we also present the well-posedness following Babuska theory~\cite{1973Babuska-a}.
Let $\bm{x} = (\bu, p) \in \bXh$ and $\bm{y} = (\bv, q) \in \bXh$, where $\bXh = \mathcal{V}_h \times \mathcal{W}_h$ is fixed\footnote{To avoid unnecessary complications in the notation, in this subsection, the pair $(\bu,p)$ denotes a generic element from~$\bXh$ instead of the solution to~\eqref{stokes1}--\eqref{bdry-t}.}. We first define the following composite bilinear form corresponding to the EG scheme, \eqref{eqn: eg}. 
\begin{equation*}
	\mathcal{L}(\bm{x}, \bm{y}) := \bah(\bu,\bv) + \bbh(\bm{v}, p) + \bbh(\bu, q).
\end{equation*}
In addition, we also define a weighted norm on the space $\bXh$ as follows, for $\bm{x} \in \bXh$,
\begin{equation*}
	\| \bm{x} \|^2_{\bXh} := \enorm{\bu}^2 + \frac{1}{2 \mu} \| p \|^2_{0}.
\end{equation*}


In the following well-posedness result, the constants $\Gamma$ and $\gamma$ only depend on the spatial dimension $d$, the coercivity constant $\chi$ in \eqref{eqn: coer}, and the inf-sup constant $\beta$ in \eqref{eqn: infsup}. The estimates are robust with respect to other parameters in the problem including $\mu$ and $h$, which also implies that the EG scheme~\eqref{eqn: eg} is well-posed. 

\begin{theorem} \label{thm:well-posedness}
	For a sufficiently large $\alpha$, the bilinear form $\mathcal{L}$ is continuous and satisfies the inf-sup condition. That is, there exist positive constants $\Upsilon$ and $\gamma$, independent of $h$, such that 
\begin{align} 
& |\mathcal{L}(\bm{x}, \bm{y})| \leq \Upsilon \| \bm{x} \|_{\bXh} \| \bm{y} \|_{\bXh}, \quad \forall \, \bm{x}, \bm{y} \in \bXh,  \label{ine:L-conti}\\
&\inf_{\bm{x} \in \bXh} \sup_{\bm{y} \in \bXh} \frac{|\mathcal{L}(\bm{x}, \bm{y})|}{\|\bm{x}\|_{\bXh} \| \bm{y} \|_{\bXh}} \geq \gamma. \label{ine:L-infsup}
	\end{align}
\end{theorem}

\begin{proof}
	For the continuity of $\mathcal{L}$, using Lemmas~\ref{lem:u-conti} and~\ref{lem:b-conti}, we have
	\begin{align*}
		|\mathcal{L}(\bm{x}, \bm{y})|
		&=| \bah(\bu,\bv) + \bbh(\bm{v}, p) + \bbh(\bu, q) | \\
		& \leq 2 \enorm{\bu} \enorm{\bv} + \frac{\sqrt{d}}{\sqrt{2\mu}} \| p \|_0 \enorm{\bv} + \frac{\sqrt{d}}{\sqrt{2\mu}} \| q \|_0 \enorm{\bu} \\
		& =
		\begin{pmatrix}
			\enorm{\bv} & \frac{1}{\sqrt{2\mu}}\|q\|_0   
		\end{pmatrix}
		\begin{pmatrix}
			2&\sqrt{d}\\
			\sqrt{d} & 0       
		\end{pmatrix}
		\begin{pmatrix}
			\enorm{\bu}\\ \frac{1}{\sqrt{2\mu}}\|p\|_0 
		\end{pmatrix}\\
		& \leq 
		\left(1+\sqrt{d+1}\right)
		\| \bm{x} \|_{\bXh} \| \bm{y} \|_{\bXh}.
	\end{align*}
	This implies the continuity of $\mathcal{L}$~\eqref{ine:L-conti} with $\Upsilon = \left(1+\sqrt{d+1}\right)$.
	
	To show the inf-sup condition, \eqref{ine:L-infsup}, for any given $\bm{x} = (\bu, p) \in \bXh$, we take $\bv = \bu + t \frac{\beta}{\sqrt{2 \mu}} \bw$ and $q = -p$, where $\bw$ is chosen based on the inf-sup condition~\eqref{eqn: infsup} such that $\bbh(\bw, p) \geq \frac{\beta}{\sqrt{2\mu}} \|p\|_0 \enorm{\bw}$. Since this inequality is invariant with respect to any scaling of $\bw$, we may choose $\bw$ such that $\enorm{\bw} = \|p\|_0$.  Then, using Lemma~\ref{lem:u-conti} and~\ref{lem:b-conti},
	\begin{align*}
		\mathcal{L}(\bm{x}, \bm{y}) &= \bah(\bu, \bu+t \frac{\beta}{\sqrt{2 \mu}} \bw) + \bbh(\bu + t \frac{\beta}{\sqrt{2 \mu}}\bw, p) - \bbh(\bu, p) \\
		& = \bah(\bu, \bu) + t \frac{\beta}{\sqrt{2 \mu}} \bah(\bu,\bw) + t \frac{\beta}{\sqrt{2 \mu}} \bbh(\bw, p) \\
		& \geq \chi \enorm{\bu}^2 - 2t \frac{\beta}{\sqrt{2 \mu}} \enorm{\bu} \enorm{\bw} + t\frac{\beta^2}{2\mu} \|p \|_0 \enorm{\bw} \\
		& = \chi \enorm{\bu}^2 - 2t \enorm{\bu} \left( \frac{\beta}{\sqrt{2 \mu}} \| p \|_0  \right)+ t  \left(\frac{\beta^2}{2\mu} \|p \|^2_0 \right) \\
		& = 
		\begin{pmatrix}
			\enorm{\bu} &
			\frac{\beta}{\sqrt{2 \mu}} \| p \|_0
		\end{pmatrix}		
		\begin{pmatrix}
			\chi & -t \\
			-t & t
		\end{pmatrix}
		\begin{pmatrix}
			\enorm{\bu} \\
			\frac{\beta}{\sqrt{2 \mu}} \| p \|_0
		\end{pmatrix}.
	\end{align*}
	If $0 < t < \chi$, the matrix is symmetric positive and definite and, therefore, there exists $\gamma_0$ such that
	\begin{equation*}
		\mathcal{L}(\bm{x}, \bm{y}) \geq \gamma_0 \left( \enorm{\bu}^2 + \frac{\beta^2}{2 \mu} \| p \|_0^2 \right) \geq \tilde{\gamma} \|\bm{x} \|_{\bXh}^2,
	\end{equation*}
	where $\tilde{\gamma} = \gamma_0 \min\{1, \beta^2\}$. In addition, it is straightforward to show that $\| \bm{y} \|_{\bXh} = \|(\bm{v}, q) \|_{\bXh} \leq \bar{\gamma} \| \bm{x} \|_{\bXh}$ for some constant $\bar{\gamma} >0$. Therefore, $\mathcal{L}$ satisfies~\eqref{ine:L-infsup} with $\gamma = \tilde{\gamma}/\bar{\gamma}$.
\end{proof}

\section{Convergence Analysis}\label{sec:convergence}
Now that we have established the well-posedness of the EG scheme, we next prove optimal error estimates for the velocity $\buh$ and the pressure $\ph$. To facilitate the analysis, we introduce the following variables:
\begin{alignat*}{2}
	& \etau = \bu - \Pih \bu, &&\quad \xiu = \Pih \bu - \buh, \\
	& \etap = p -   \calP_0 p, &&\quad \xip =  \calP_0 p - \ph.
\end{alignat*}
The interpolation error estimates for $\etau$ and $\etap$ are well-known, so we need only to prove the error estimates for the auxiliary variables, $\xiu$ and $\xip$. 
From the consistency equations, \eqref{eqn: consis}, we derive the following error equations:
\begin{subequations}\label{eqn: error0}
	\begin{alignat}{3}
		\bah(\xiu, \bv) +\bbh(\bv, \xip) &  = - \bah(\etau, \bv) - \bbh(\bv,\etap) , && \quad  \forall \bv \in  \cVh, \label{error0-1} \\
		\bbh(\xiu, q ) &  =  -\bbh(\etau, q ), && \quad \forall q \in \cWh. \label{error0-2}
	\end{alignat}
\end{subequations}
By summing~\eqref{error0-1} and~\eqref{error0-2}, we rewrite the error equations in terms of the composite bilinear form, $\mathcal{L}$, as follows: 
\begin{equation*}
	\mathcal{L}(\xi_{\bm{x}}, \by) = - \mathcal{L}(\eta_{\bm{x}}, \by), \quad \forall\, \by = (\bv, q) \in \bXh,
\end{equation*}
where $\xi_{\bm{x}} = (\xiu, \xip) \in \bXh$ and $\eta_{\bm{x}} = (\etau, \etap)$. Note that this also implies the consistency of $\mathcal{L}$. 

\begin{lemma}\label{lem:error-xi}
Assume that the exact solution, $(\bu, p)$, belongs to $\Htwod \times \Hone$. If $\alpha$ is large enough, we have
\begin{equation*}
\left(\enorm{\xiu}^2 + \frac{1}{2\mu} \| \xi_p \|^2_0 \right)^{\half} = \| \xi_{\bm{x}} \|_{\bXh}   \leq  C \gamma^{-1} h \left( \sqrt{\mu(d+1)}\| \bu \|_2 + \sqrt{\frac{d}{\mu}} \| p \|_1  \right),
\end{equation*}
where $C>0$ is independent of $\mu$ and $h$.
\end{lemma}
\begin{proof}
By~\eqref{ine:L-infsup}, we have
\begin{align}
\left( \enorm{\xiu}^2 + \frac{1}{2\mu} \| \xi_p \|^2_0 \right)^{\frac{1}{2}} = \| \xi_{\bm{x}} \|_{\bXh} \leq \gamma^{-1} \sup_{\by \in \bXh} \frac{|\mathcal{L}(\xi_{\bm{x}}, \by)|}{\| \by \|_{\bXh}} = \gamma^{-1} \sup_{\by \in \bXh} \frac{|\mathcal{L}(\eta_{\bm{x}}, \by)|}{\| \by \|_{\bXh}}. \label{ine:error-bound}
\end{align}
Note that 
\begin{equation*}
|\mathcal{L}(\eta_{\bm{x}}, \by)| \leq |\bah(\eta_{\bu}, \bv)| + | b(\bv, \eta_p)|  + | b(\eta_{\bu}, q) |.
\end{equation*} 
		
To bound the first term, $|\bah(\eta_{\bu}, \bv)|$, we follow the same steps shown in the proof of Lemma~\ref{lem:u-conti} and use the trace inequality \eqref{eqn:trace1}:
\begin{align*}
|\bah(\eta_{\bu}, \bv)| &\leq 2 \mu\left ( \sumK \norm{\epsilon(\eta_{\bu})}_{0,T}^2 + \sumEod \frac{h_e}{\alpha} \| \av{\epsilon(\eta_{\bu})}\bne \|_{0,e}^2 + \left(1+|\theta|\right) \sumEod \frac{\alpha}{h_e} \norm{\jump{\eta_{\bu}}}_{0, e}^2  \right )^{\half} \cdot \\
& \quad \quad \   \left ( \sumK \norm{\epsilon(\bv)}_{0,T}^2 + |\theta|\sumEod \frac{h_e}{\alpha} \| \av{\epsilon(\bv)}\bne \|_{0, e}^2 + 2 \sumEod \frac{\alpha}{h_e} \norm{\jump{\bv}}_{0, e}^2 \right )^{\half}\\
& \leq 2\mu \left (\sumK \norm{\eta_{\bu}}_{1,\K}^2  + \frac{C}{\alpha} \sumK \left( \norm{\eta_{\bu}}^2_{1,\K} + h_{\K}^2 \norm{\eta_{\bu}}^2_{2,\K}  \right)  + C \alpha \sumK  \frac{1}{h_T^2} \left( \norm{\eta_{\bu}}^2_{0,T} + h_T^2 \| \eta_{\bu} \|^2_{1,\K} \right)  \right )^{\half} \cdot \\
& \quad \quad \ \left ( \left(1 + \frac{C}{\alpha}\right) \sum_{\K \in \Th} \norm{\epsilon(\bv)}_{0, \K}^2  +  2 \sumEod \frac{\alpha}{h_e} \norm{\jump{\bv}}_{0, e}^2 \right)^{1/2} \\
& \leq C h   \sqrt{\mu} \norm{\bu}_2  \enorm{\bv}.
\end{align*}
For the second term, mimicking the steps in the proof of Lemma~\ref{lem:b-conti}, we have
\begin{align*}
& |\bbh(\bv, \eta_p)| \\
& \leq \frac{1}{\sqrt{2\mu}} \| \eta_p \|_0 \, \sqrt{2\mu} \left( \sum_{\K \in \Th} \| \divv \bv \|^2_{0, \K} \right)^{\half} + \left( \frac{1}{2\mu}\sumEod \frac{h_e}{\alpha} \|\av{\eta_p} \bne \|_{0,e}^2 \right)^{\half} \left( 2 \mu \sumEod \frac{\alpha}{h_e} \| \jump{\bv} \|_{0,e}^2 \right)^{\half}  \\
& \leq \frac{1}{\sqrt{2\mu}} \norm{\eta_p}_0 \left( 2\mu \, d  \sumK \norm{\epsilon(\bv)}_{0,T}^2  \right)^{\half} + \frac{1}{\sqrt{2 \mu}} \frac{\sqrt{C}}{\sqrt{\alpha}} \left( \sumK (\norm{\eta_p}_{0,T}^2 + h_T^2 \norm{\eta_p}_{1,T}) \right)^{\half} \left( 2 \mu \sumEod \frac{\alpha}{h_e} \norm{\jump{\bv}}^2_{0,e} \right)^{\half} \\
& \leq C h \left(\frac{\sqrt{d}}{\sqrt{\mu}}  \norm{p}_1 \right) \enorm{\bv}.
\end{align*}
Similarly, we obtain the following estimate for the third term,
\begin{align*}
|\bbh(\eta_{\bu}, q)| &\leq C h \left( \sqrt{\mu d} \norm{\bu}_2 \right)\left( \frac{1}{\sqrt{\mu}} \| q\|_0 \right).
\end{align*}
Therefore, 
\begin{align*}
|\mathcal{L}(\eta_{\bm{x}}, \by)| & \leq 	Ch \left[  \left( \sqrt{\mu} \norm{\bu}_2 \right) \enorm{\bv} +  \left( \frac{\sqrt{d}}{\sqrt{\mu}} \|p \|_1 \right) \enorm{\bv} +  \left( \sqrt{\mu d} \norm{\bu}_2\right) \left( \frac{1}{\sqrt{\mu}}  \| q\|_0 \right) \right] \\
&\leq  Ch \left( \mu \| \bu \|_2^2 + \frac{d}{\mu} \|p\|_1^2 + \mu d \|\bu \|_2^2 \right)^{\half} \left( \enorm{\bv}^2 + \enorm{\bv}^2 + \frac{1}{\mu} \|q\|_0^2 \right)^{\half} \\
& \leq Ch \left( \mu (d+1) \| \bu \|_2^2 + \frac{d}{\mu} \| p \|_1^2  \right)^{\half} \norm{\by}_{\bXh}.
\end{align*}
Substituting the above inequality back into~\eqref{ine:error-bound} completes the proof.
\end{proof}

Now, based on the interpolation error estimates for $\eta_{\bu}$ and $\eta_p$, we immediately have the following convergence results of the EG scheme, \eqref{eqn: eg}.

\begin{theorem}
Let $(\bu, p)$ be the solution to \eqref{eqn: stokes} and $(\buh,\ph)$ be the solution to the EG scheme in \eqref{eqn: eg}.
If we assume that $(\bu, p)$ belongs to $\Htwod \times \Hone$, then we have the following error estimates:
\begin{align}
\enorm{\bu - \buh} &\leq C h \left(  \sqrt{\mu(d+1)} \norm{\bu}_2 + \sqrt{\frac{d}{\mu}}\norm{p}_1\right), \label{err-u}\\
\norm{p - \ph}_0 &\leq C h  \left( \mu \sqrt{(d+1)} \norm{\bu}_2 + (1 + \sqrt{d}) \norm{p}_1\right), \label{err-p}
\end{align}
where $C >0$ is independent of $h$ and $\mu$.
\end{theorem}
\begin{proof}
Error estimates~\eqref{err-u} and~\eqref{err-p} are direct consequences of Lemma~\ref{lem:error-xi}, interpolation error estimates of $\eta_{\bu}$ and $\eta_p$, and the triangle inequality.
\end{proof}

\section{Block Preconditioners}\label{sec:block-preconditioners}
When discretizing the Stokes equations, we obtain a large-scale, ill-conditioned linear system. We consider iterative solution techniques and use Krylov subspace methods to solve the system of equations.  In order to accelerate the convergence of Krylov subspace methods, following  the general framework developed in~\cite{LonghinWathen,MardalWinther,ma2016robust}, we develop robust block preconditioners by taking advantage of the well-posedness of the proposed EG discretization. 

The  linear system resulting from the EG method for solving \eqref{eqn: eg} can be written in the following two-by-two block form:
\begin{equation} \label{eqn:linear-system}
	\mathcal{A}\bm{x} = \bm{b}, \quad \mathcal{A} = \begin{pmatrix}
		\bm{A} &  \bm{B}^T \\
		\bm{B} &   \bm{0} 
	\end{pmatrix},
	\quad 
	\bm{x} = 
	\begin{pmatrix}
		\bu_h \\ 
		p_h
	\end{pmatrix},
	\ \text{and} \
	\bm{b} = 
	\begin{pmatrix}
		\bm{f} \\
		\bm{g}
	\end{pmatrix},
\end{equation}
where $\ba_{\theta}(\bu, \bv) \rightarrow \bm{A}$, $\bb(\bv, p) \rightarrow \bm{B}$, $\bF_\theta(\bv) \rightarrow \bm{f}$, and $\sum_{e \in \mathcal{E}_h^{\partial, d}}(\{ w \},  \jump{\bg} \cdot \bn_e)_e \rightarrow \bm{g}$ represent the discrete version of the weak formulation.

\subsection{Block Diagonal Preconditioner}
Based on the well-posedness result (Theorem~\ref{thm:well-posedness}) and the framework proposed in~\cite{LonghinWathen,MardalWinther}, a natural choice for a \emph{norm-equivalent} preconditioner is the Riesz operator with respect to the inner product corresponding to the weighted norm $\| \cdot \|_{\bm{X}_h}$.  This leads to the following block diagonal preconditioner in matrix form:
\begin{equation*}
	\mathcal{B}_D = 
	\begin{pmatrix}
		\bm{A}_{\mathcal{E}} & 0 \\
		0  &  \frac{1}{2\mu}\bm{M}_p
	\end{pmatrix}^{-1},
\end{equation*}
where $(\bu_h, \bv_h)_{\mathcal{E}} \rightarrow \bm{A}_{\mathcal{E}}$ and $(p_h, q_h) \rightarrow \bm{M}_p$. Note that, even for different choices of $\theta$, $\bm{A}_{\mathcal{E}}$ is SPD since it corresponds to the energy norm. In practice, applying the preconditioner $\mathcal{B}_D$ involves inverting the diagonal blocks, which could be expensive and sometimes infeasible for large-scale ill-conditioned problems. Therefore, we replace the inverse of the diagonal blocks of $\mathcal{B}_D$ by their spectral equivalent SPD approximations and define an inexact block diagonal preconditioner,
\begin{equation*}
	\mathcal{M}_D = 
	\begin{pmatrix}
		\bm{H}_{\mathcal{E}} & 0 \\
		0  &  \bm{H}_p
	\end{pmatrix},
\end{equation*}
where
\begin{align}
	&c_{1,\bu} (\bm{H}_{\mathcal{E}} \bu, \bu) \leq (\bm{A}_{\mathcal{E}}^{-1} \bu, \bu) \leq c_{2,\bu}  (\bm{H}_{\mathcal{E}} \bu, \bu),  \label{ine:u-equiv-blk}\\
	& c_{1,p} (\bm{H}_{p} p, p) \leq (2 \mu \bm{M}_{p}^{-1} p, p) \leq c_{2,p}  (\bm{H}_{p} p, p), \label{ine:p-equiv-blk}
\end{align}
for some constants $c_{1,\bu}, \ c_{2,\bu}, \ c_{1,p}, \ c_{2,p}>0$ that are independent of $h$ and $\mu$.  For example, we can use the method developed in~\cite{YiLeeZikatanov21} or a multigrid method to define $\bm{H}_{\mathcal{E}}$ and solve $\bm{A}_{\mathcal{E}}$.  Since $\bm{M}_p$ is diagonal, it can be inverted exactly and, therefore, we use $\bm{H}_p = 2\mu \bm{M}_p^{-1}$. Both $\mathcal{B}_D$ and $\mathcal{M}_D$ can be applied to the minimal residual (MINRes) method or the generalized MINRes (GMRes) method and the resulting iterative method is parameter-robust based on the general framework in~\cite{LonghinWathen,MardalWinther}. We summarize the results in the following theorem.
\begin{theorem}\label{thm:blk-diag-prec}
	If $\alpha$ is sufficiently large and \eqref{ine:u-equiv-blk} and~\eqref{ine:p-equiv-blk} hold with constants independent of parameters $\mu$ and $h$, then,
	\begin{equation*}
		\kappa(\mathcal{B}_D \mathcal{A}) = \mathcal{O}(1) \quad \text{and} \quad \kappa(\mathcal{M}_D \mathcal{A}) = \mathcal{O}(1), \label{eqn:blk-diag-prec-robust}
	\end{equation*}
where $\kappa(\cdot)$ is the condition number of a matrix.
\end{theorem}

\subsection{Block Triangular Preconditioner}
Similarly, following the framework presented in~\cite{LonghinWathen,adler2017robust,2020AdlerGasparHuOhmRodrigoZikatanov-a}, we can also develop block triangular preconditioners (field-of-value-equivalent preconditioners). Based on the well-posedness, Theorem~\ref{thm:well-posedness}, and the Riesz operator, $\mathcal{B}_D$, we consider the following block lower triangular preconditioner,
\begin{equation*} 
	\mathcal{B}_L = 
	\begin{pmatrix}
		\bm{A}_{\mathcal{E}} & 0 \\
		\bm{B}  &  \frac{1}{2 \mu} \bm{M}_p
	\end{pmatrix}^{-1}
\end{equation*}
and upper triangular preconditioner,
\begin{equation*} 
	\mathcal{B}_U = 
	\begin{pmatrix}
		\bm{A}_{\mathcal{E}} & \bm{B}^T \\
		0  &  \frac{1}{2 \mu} \bm{M}_p
	\end{pmatrix}^{-1}.
\end{equation*}
  Again, in practice, the inversion of $\bm{A}_{\mathcal{E}}$ can be defined by the method proposed in \cite{YiLeeZikatanov21} or a multigrid method, while $\bm{M}_p^{-1}$ is computed exactly. These approaches provide spectral equivalent symmetric and positive definite approximations as shown in~\eqref{ine:u-equiv-blk} and~\eqref{ine:p-equiv-blk}, and they define the following inexact block triangular preconditioners:
\begin{equation*}
	\mathcal{M}_L = 
	\begin{pmatrix}
		\bm{H}_{\mathcal{E}}^{-1} & 0 \\
		\bm{B}  & \bm{H}_p^{-1}
	\end{pmatrix}^{-1}
	\quad \text{and} \quad 
	\mathcal{M}_U = 
	\begin{pmatrix}
		\bm{H}_{\mathcal{E}}^{-1} & \bm{B}^T \\
		0  &  \bm{H}_p^{-1}
	\end{pmatrix}^{-1}.
\end{equation*}

Following~\cite{LonghinWathen,adler2017robust,2020AdlerGasparHuOhmRodrigoZikatanov-a}, we show that the block triangular preconditioners are field-of-value-equivalent with $\mathcal{A}$. Due to the length constraint of this paper and the fact the proofs are similar to those in~\cite{adler2017robust,2020AdlerGasparHuOhmRodrigoZikatanov-a}, we only state the results here. 

\begin{theorem}\label{thm:blk-tri-prec}
	Assume $\alpha$ is sufficiently large and~\eqref{ine:u-equiv-blk} and~\eqref{ine:p-equiv-blk} hold with constants independent of parameters $\mu$ and $h$.  Furthermore, assume that $\| \bm{I} - \bm{H}_{\mathcal{E}} \bm{A}_{\mathcal{E}} \|_{\mathcal{E}} \leq \rho$, $0 \leq  \rho <1$.  Then, there exist constants $\Upsilon_1$ and $\Upsilon_2$, independent of discretization or physical parameters, such that, for $\bm{0} \neq \bm{x} \in \bX_h$ or $\bm{0} \neq \bx' \in \bXh'$,
	\begin{align*}
		&\Upsilon_1 \leq \frac{(\mathcal{B}_L\mathcal{A}\bx, \bx)_{\mathcal{B}_D^{-1}}}{(\bx,\bx)_{\mathcal{B}_D^{-1}}}, \quad \frac{\| \mathcal{B}_L \mathcal{A} \bx \|_{\mathcal{B}_D^{-1}} }{\| \bx \|_{\mathcal{B}_D^{-1}}} \leq \Upsilon_2, \\
		&\Upsilon_1 \leq \frac{(\mathcal{M}_L\mathcal{A}\bx, \bx)_{\mathcal{M}_D^{-1}}}{(\bx,\bx)_{\mathcal{M}_D^{-1}}}, \quad \frac{\| \mathcal{M}_L \mathcal{A} \bx \|_{\mathcal{M}_D^{-1}} }{\| \bx \|_{\mathcal{M}_D^{-1}}} \leq \Upsilon_2, \\
		&\Upsilon_1 \leq \frac{(\mathcal{A}\mathcal{B}_U\bx', \bx')_{\mathcal{B}_D}}{(\bx',\bx')_{\mathcal{B}_D}}, \quad \frac{\| \mathcal{A}\mathcal{B}_U  \bx \|_{\mathcal{B}_D} }{\| \bx' \|_{\mathcal{B}_D}} \leq \Upsilon_2, \\
		&\Upsilon_1 \leq \frac{(\mathcal{A}\mathcal{M}_U\bx', \bx')_{\mathcal{M}_D}}{(\bx',\bx')_{\mathcal{M}_D}}, \quad \frac{\| \mathcal{A}\mathcal{M}_U  \bx \|_{\mathcal{M}_D} }{\| \bx' \|_{\mathcal{M}_D}} \leq \Upsilon_2.
	\end{align*}
Here, given a symmetric positive definite matrix $\mathcal{D}$, $(\bx, \by)_{\mathcal{D}}:= (\mathcal{D}\bx, \by)$ denotes its induced inner product and the corresponding norm is defined as $\| \bx \|_{\mathcal{D}} := \sqrt{(\bx, \bx)_{\mathcal{D}}}$.
\end{theorem}
Such block triangular preconditioners are applied to GMRes and, following the framework in~\cite{LonghinWathen,adler2017robust}, provide parameter-robust iterative methods for solving the proposed EG discretizations.

\section{Numerical Results}\label{sec:numerical-results}
We now present several numerical examples in order to validate the convergence analysis and the performance of the proposed linear solver with the preconditioner. In addition, mixed boundary conditions and different viscosities  are tested to investigate the capability of our method. For all examples in this section, we focus on the IIPG method ($\theta =0$) because it exposes the difficulties in approximating the solution to the Stokes equations and in preconditiong the resulting non-symmetric linear system robustly.  All numerical experiments, including the implementation of EG discretization and the  preconditioned linear solvers, were implemented using the finite-element and solver library HAZmath \cite{2014AdlerJ_HuX_ZikatanovL-aa}.

For testing the performance of the block preconditioners, flexible GMRes is used to solve the linear systems~\eqref{eqn:linear-system} obtained by the EG discretization~\eqref{eg1}-\eqref{eg2}. 
A stopping tolerance of $10^{-6}$ is used for the relative residual ($\frac{\| \bm{b} - \mathcal{A} \bm{x}^k \|}{\| \bm{b} \|}$) of the linear system~\eqref{eqn:linear-system}, where $\bm{x}^k$ is the $k$-th iteration of the flexible GMRes method. For the implementation of $\mathcal{B}_D$, $\mathcal{B}_L$, and $\mathcal{B}_U$, we use $\bm{A}$ instead of $\bm{A}_{\mathcal{E}}$ as the first diagonal block in order to make our preconditioner more user-friendly (note that $\bm{A}_{\mathcal{E}}$ needs to be assembled separately) and then solve it exactly by using the UMFPACK library~\cite{davis2004algorithm}.  Based on Lemma~\ref{lem: u-coer} and~\ref{lem:u-conti} and Theorem~\ref{thm:blk-diag-prec} and~\ref{thm:blk-tri-prec}, this still gives robust preconditioners which will be confirmed by our numerical experiments. For the implementation of $\mathcal{M}_D$, $\mathcal{M}_L$, and $\mathcal{M}_U$, the action of $\bm{A}^{-1}$ is approximated by the GMRes method with the preconditioner proposed in~\cite{YiLeeZikatanov21}, which was specially designed for the EG scheme. To be more precise, we use a multiplicative version preconditioner in which the continuous linear element part is solved by one step of a V-cycle-smoothed aggregation algebraic multigrid method, and the Jaocbi method is used as a global smoother.  Since we only need to solve the diagonal block $\bm{A}$ approximately, a relative residual tolerance of $10^{-3}$ is used for the preconditioned GMRes method. On the other hand, since we use a piecewise constant space for the pressure, $\bm{M}_p$ is a diagonal matrix which can be inverted easily.  Therefore, we just invert the second diagonal block $\frac{1}{2\mu} \bm{M}_p$ exactly in all the block preconditioners. 

\subsection{Example 1: Homogeneous Dirichlet Boundary Conditions} 
\label{sec:ex1}
First, we consider the Stokes system \eqref{eqn: stokes}, with homogenous Dirichlet boundary conditions defined on a computational domain, $\Omega = [0,1]^2$. 
The divergence-free ($\nabla \cdot \bu = 0$), exact solutions are given as 
\begin{equation*}
	\bu = 
	\begin{pmatrix}
		u_1 \\
		u_2 
	\end{pmatrix}
	=
	\begin{pmatrix}
		\sin(\pi x)  \sin(\pi y) \\
		\cos(\pi x)  \cos(\pi y) 
	\end{pmatrix},
	\ \  p = \sin(\pi x) \cos(\pi y) 
\end{equation*}
with $\mu = 1$. 

The error is measured in the energy norm for the velocity and in the $L^2$-norm for the pressure on uniform meshes with various mesh sizes $h=2^{-L}, L=2, \hdots, 6$. 
The penalty parameter is set to $\alpha=1$. 
The results of error computations are summarized in Table~\ref{tab:tab1}, where we observe optimal convergence rates in the both velocity and  pressure. 
Next, Table~\ref{tab:prec-ex1} shows iterations counts for the block preconditioners with different mesh sizes.  The number of iterations are relatively consistent for all cases, which shows that the proposed block preconditioners are robust with respect to the discretization parameters.  The block lower and upper triangular preconditioners perform slightly better than the block diagonal preconditioners as expected, since they contain more coupling information of the original system. 

\begin{table}[!h]
	\centering
	\begin{tabular}{|c||c|c|c||c|c|c|}
		\hline
		$h$    & DoFs & $\enorm{ u -U}$ & Rate & DoFs & $ \| p - P \|_0$ & Rate  \\ \hline
		1/4 & 82 &  1.3624 &  0.00 & 32  & 1.1553  & 0.00  \\ \hline 
		1/8 & 290 &  0.6706 &  1.12 & 128  & 0.4991  & 1.21  \\ \hline 
		1/16 & 1090 &  0.3206 &  1.11 & 512  & 0.1914  & 1.38  \\ \hline 
		1/32 & 4226 &  0.1545 &  1.07 & 2048  & 0.0726 & 1.39  \\ \hline 
		1/64 & 16642 &  0.0756 &  1.04 & 8192  & 0.0286  & 1.34  \\ \hline 
	\end{tabular}
	\caption{Example 7.1:
Convergence study for the IIPG-EG method ($\theta=0$). 
The penalty parameter is set to $\alpha=1$. Here, DoFs refers to the total number of degrees of freedom for each space. }
	\label{tab:tab1}
\end{table}

\begin{table}[h!]
	\begin{center}
		\begin{tabular}{| c || c | c | c || c | c | c |}
			\hline 
			$h$ & $\mathcal{B}_D$ & $\mathcal{B}_L$ & $\mathcal{B}_U$  & $\mathcal{M}_D$ & $\mathcal{M}_L$ & $\mathcal{M}_U$ 
			\\ 
			\hline 
			1/8	& 22 & 11 & 11 & 22 & 14 & 17  \\
			1/16 & 24 & 12 & 12 & 24 & 15 & 18 \\
			1/32 & 24 & 11 & 11 & 24 & 15 & 18  \\
			1/64 & 22 & 11 & 10 & 24 & 14 & 20  \\
			\hline 
		\end{tabular} 
	\end{center}
    \caption{Example~\ref{sec:ex1}. Iteration counts for the block preconditioners.}
	\label{tab:prec-ex1}
\end{table}

\subsection{Example 2: Mixed Boundary Conditions}  
\label{sec:ex2}
Next, we consider an example with mixed boundary conditions.
Dirichlet boundary conditions are imposed on the left and the right side of the boundary, $\partial \Omega$ (i.e., $\bu = \bg$ on $x=0$ and $x=1$), and Neumann boundary conditions are applied on the top and the bottom of the boundary, $\partial \Omega$ (i.e., $(2\mu \epsilon(\bu) - p \bI) \bn = \bs$ on $y=0$ and $y=1$).
The rest of the setup is the same as in Example 1, and  $\bg$ and $\bs$ are computed from the given exact solutions.  

The computed results are presented in Table~\ref{tab:tab2}, and we observe that the EG method yields optimal convergence rates in both the velocity and the pressure.
In addition,  Table~\ref{tab:prec-ex2} results for the block preconditioners yield similar results to the previous example.  Thus, the proposed block preconditioners are robust for this example as well.
\begin{table}[!h]
	\centering
	\begin{tabular}{|c||c|c|c||c|c|c|}
		\hline
		$h$    & DoFs & $\enorm{ u -U}$ & Rate & DoFs & $ \| p - P \|_0$ & Rate  \\ \hline
		1/4 & 82 &  1.4728&  0.00 & 32  & 0.7767  & 0.00  \\ \hline 
		1/8 & 290 &  0.6761 &  1.23 & 128  & 0.3554  & 1.12  \\ \hline 
		1/16 & 1090 &  0.3165 &  1.14 & 512  & 0.1406  & 1.33  \\ \hline 
		1/32 & 4226 &  0.1526 &  1.07 & 2048  & 0.0572  & 1.29  \\ \hline 
		1/64 & 16642 &  0.0750 &  1.03 & 8192  & 0.0246  & 1.21  \\ \hline 
	\end{tabular}
	\caption{Example 7.2: Convergence study for the IIPG-EG method ($\theta =0$) with mixed boundary conditions. The penalty parameter is set to $\alpha=1$. Here, DoFs refers to the total  number of degrees of freedom for each space. }
	\label{tab:tab2}
\end{table}

\begin{table}[!h] 
	\begin{center}
		\begin{tabular}{| c || c | c  | c || c | c | c |}
			\hline 
			$h$ & $\mathcal{B}_D$ & $\mathcal{B}_L$ & $\mathcal{B}_U$  & $\mathcal{M}_D$ & $\mathcal{M}_L$ & $\mathcal{M}_U$ 
			\\ 
			\hline 
			1/8 	& 20 &  9 &  9 & 20 & 14 &  9  \\
			1/16 & 20 & 10 &  9 & 22 & 14 &  9 \\
			1/32 & 20 & 10 &  9 & 22 & 14 &  9  \\
			1/64 & 20 &  9 &  8 & 25 & 14 &  9  \\
			\hline 
		\end{tabular}
	\end{center}
    \caption{Example~\ref{sec:ex2}. Iteration counts for the block preconditioners.}
	\label{tab:prec-ex2}
\end{table}

\newpage
\subsection{Example 3: Channel Flow with Varying Viscosities} \label{sec:ex3}
\begin{wrapfigure}{r}{0.25\textwidth}
	\begin{center}
		\includegraphics[width=0.25\textwidth]{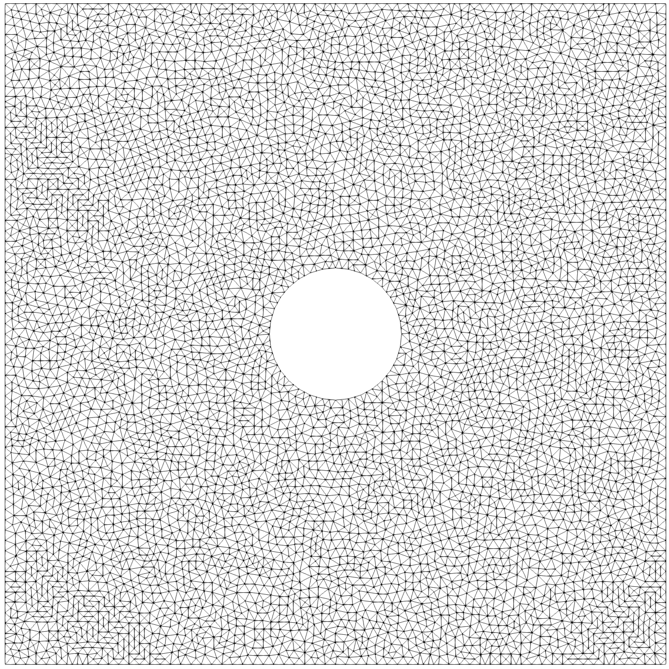}
	\end{center}
	\caption{Example 3. The domain}
	\label{fig1}
\end{wrapfigure}

In this example, we test a two-dimensional channel flow around a circular obstacle in the computational domain, 
$\Omega = [0,1]^2$.  A circular hole is centered at $(0.5, 0.5)$ with radius $0.1$ as shown in  Figure \ref{fig1}. 
We impose non-homogeneous boundary conditions as 
\begin{equation*}
	\bg= 
	\begin{cases}
		(4y(1-y),0)^{T} \ \text{ if } x = 0, \\
		(4y(1-y),0)^{T} \ \text{ if } x = 1, \\
		(0,0)^{T} \ \text{ elsewhere}, 
	\end{cases}
	\label{eqn:ex3}
\end{equation*}
and we test with several different viscosities: 
$\mu=0.001, 0.01, 0.1$, and  $1$. For all cases, the minimum mesh size is $h=2^{-7}$, and the number of degrees of freedom for the velocity and the pressure are $25,780$ and $12,718$, respectively.  The penalty coefficient is set to be $\alpha = 1$.

The vector fields of the velocity, streamlines of the velocity, and pressure values are illustrated in Figures \ref{fig:ex4_1000}-\ref{fig:ex4_1}, corresponding to the increasing values of $\mu$ from $0.001$ to $1$. 
We observe less turbulent behavior in the velocity as $\mu$ increases (equivalent to the Reynolds number getting smaller) as  expected.

\begin{figure}[!h]
	\centering
	\subfloat[Velocity $\bu$]{
		\includegraphics[width=0.33\textwidth]{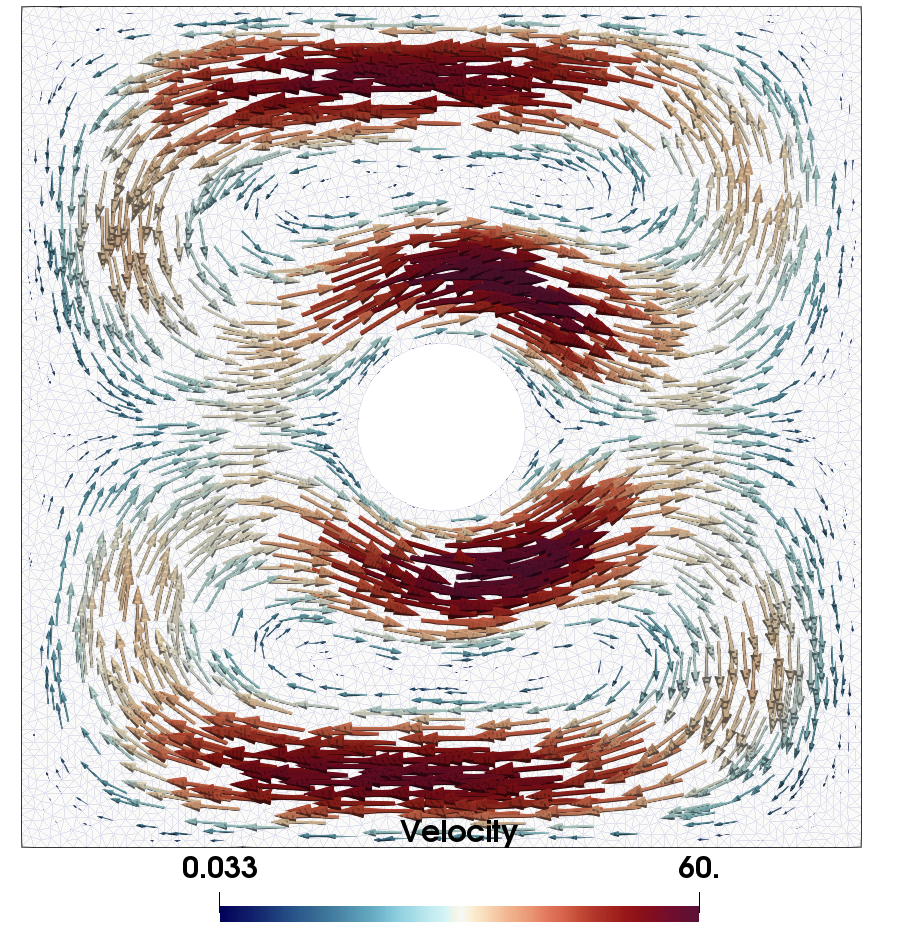}
	}
	\subfloat[Streamlines of $\bu$]{
		\includegraphics[width=0.33\textwidth]{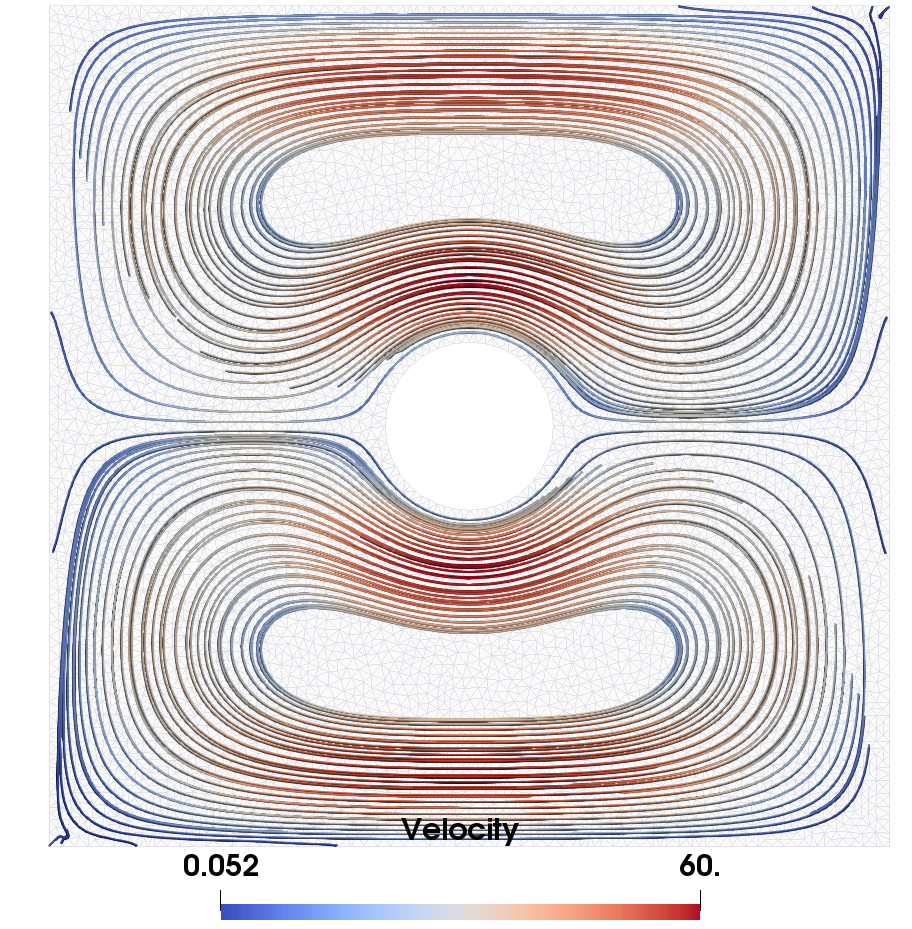}
	}
	\subfloat[Pressure p]{
		\includegraphics[width=0.33\textwidth]{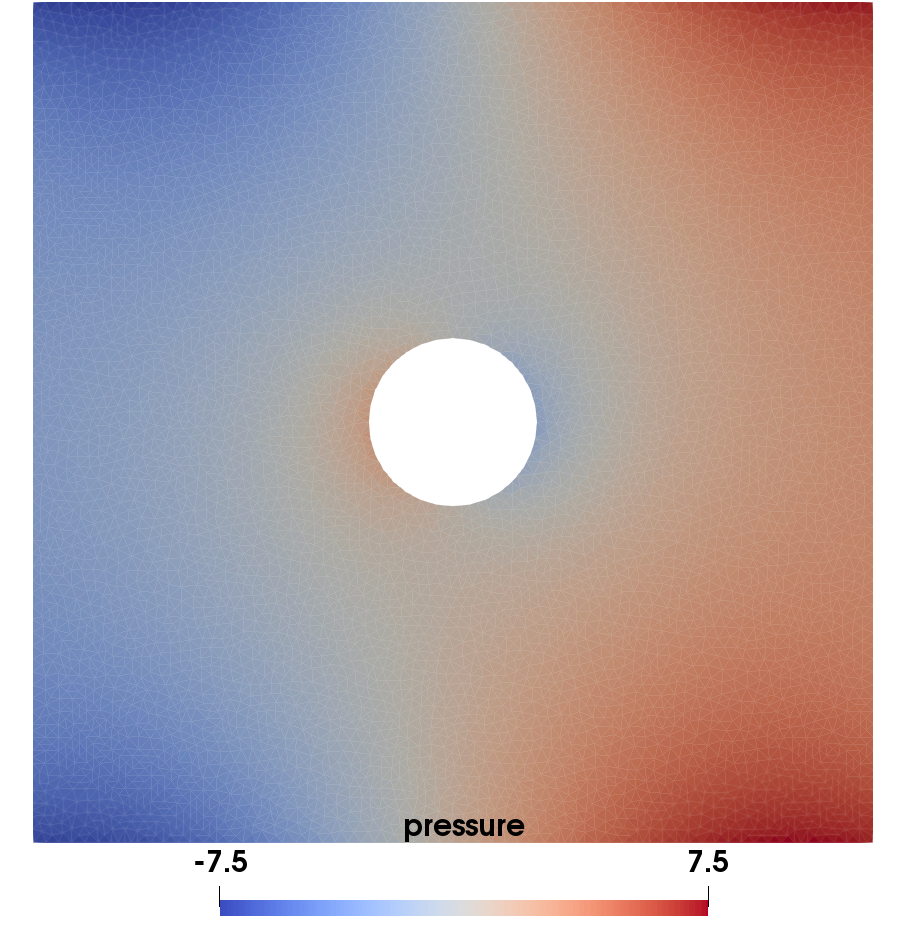}
	}
	\caption{Example 7.3: The velocity and pressure profiles for $\mu=0.001$.}
	\label{fig:ex4_1000}
\end{figure}

\begin{figure}[!h]
	\centering
	\subfloat[Velocity $\bu$]{
		\includegraphics[width=0.33\textwidth]{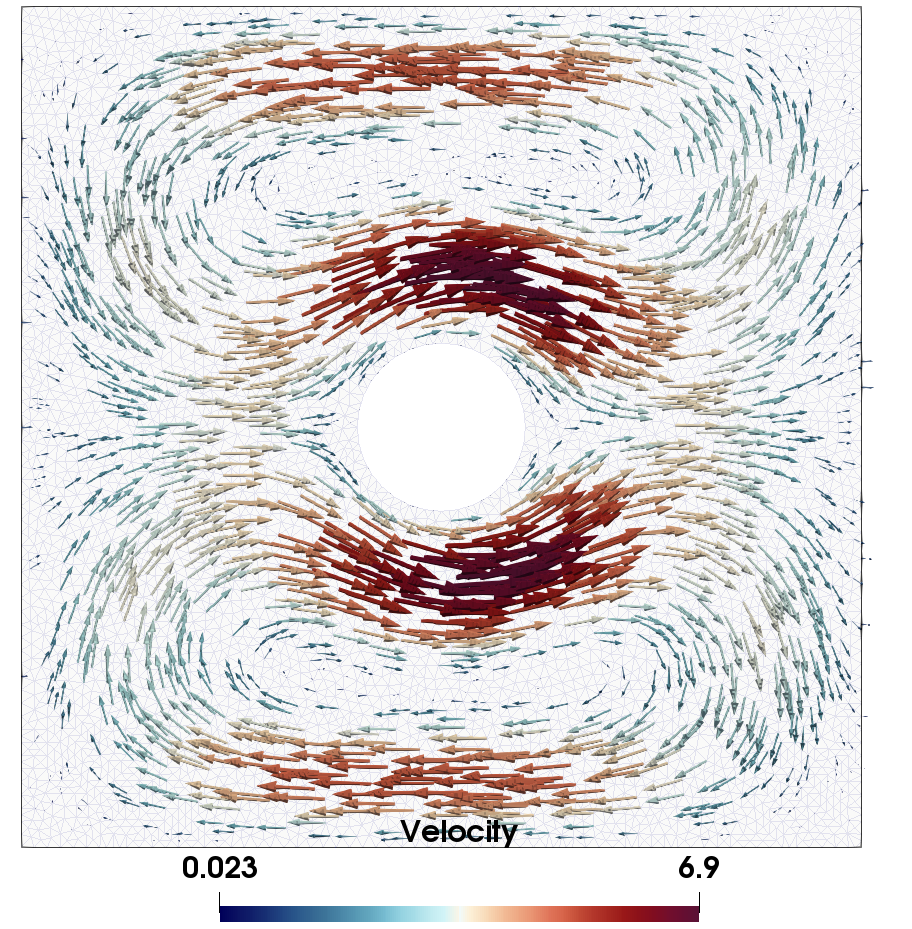}
	}
	\subfloat[Streamlines of $\bu$]{
		\includegraphics[width=0.33\textwidth]{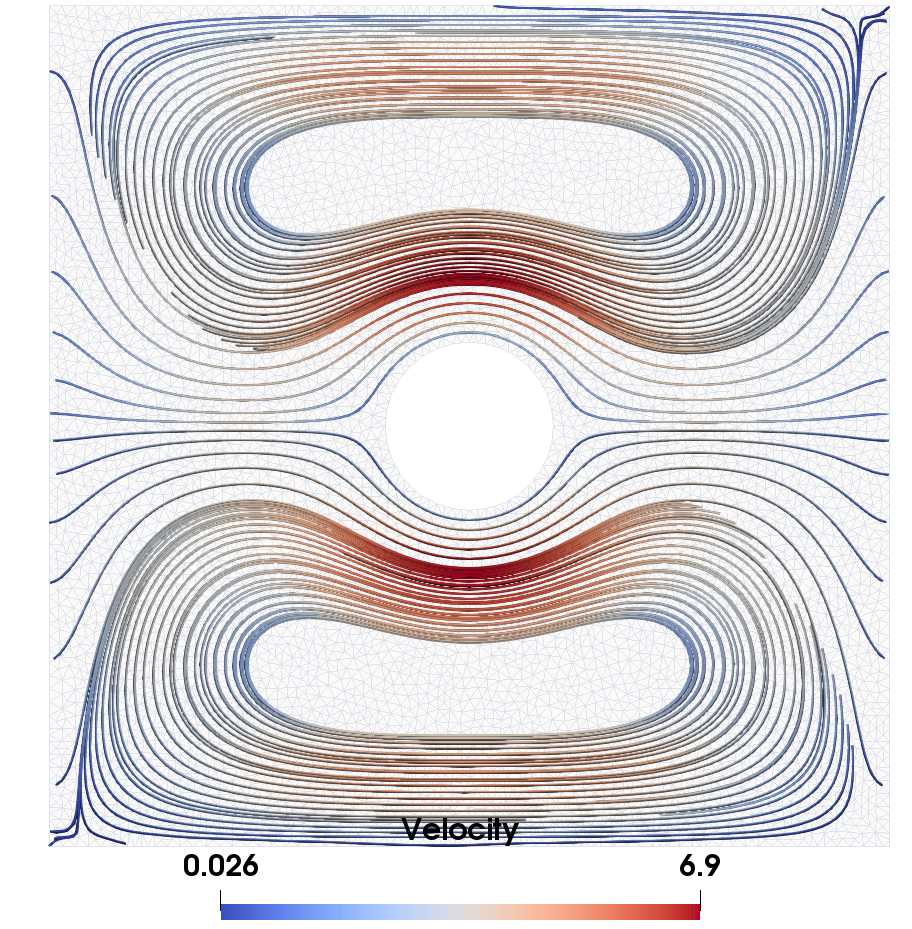}
	}
	\subfloat[Pressure p]{
		\includegraphics[width=0.33\textwidth]{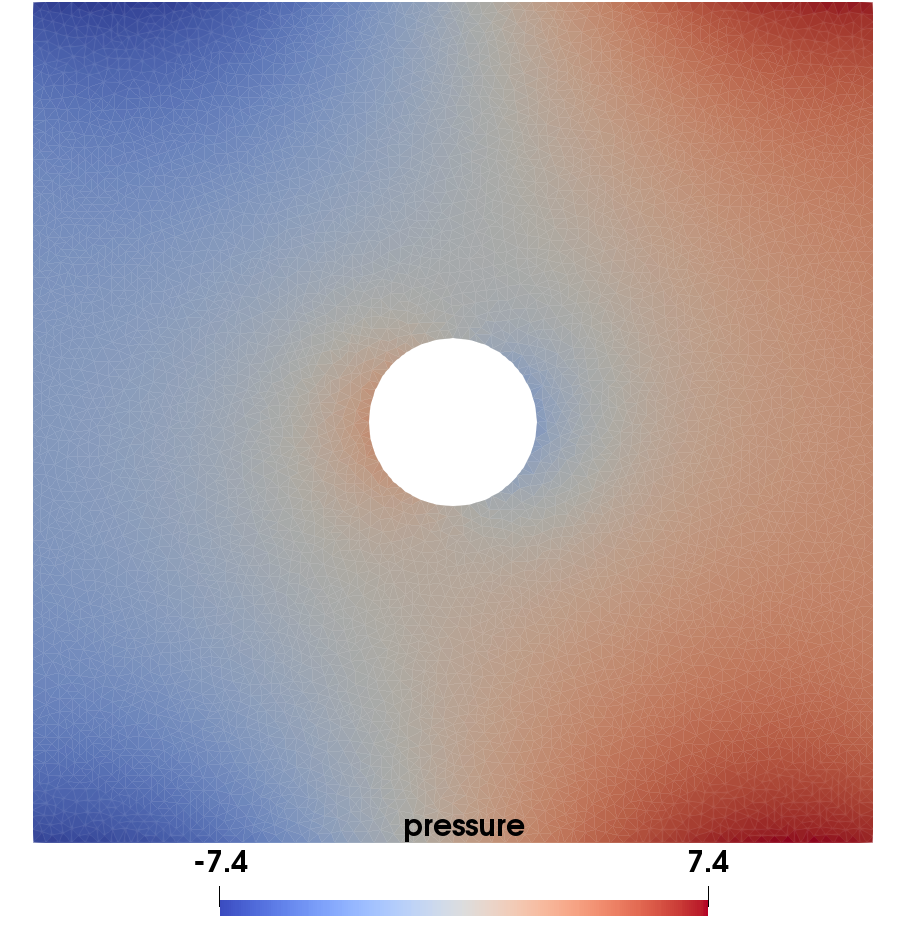}
	}
	\caption{Example 7.3: The velocity and pressure profiles for $\mu=0.01$.}
	\label{fig:ex4_100}
\end{figure}
\begin{figure}[!h]
	\centering
	\subfloat[Velocity $\bu$]{
		\includegraphics[width=0.33\textwidth]{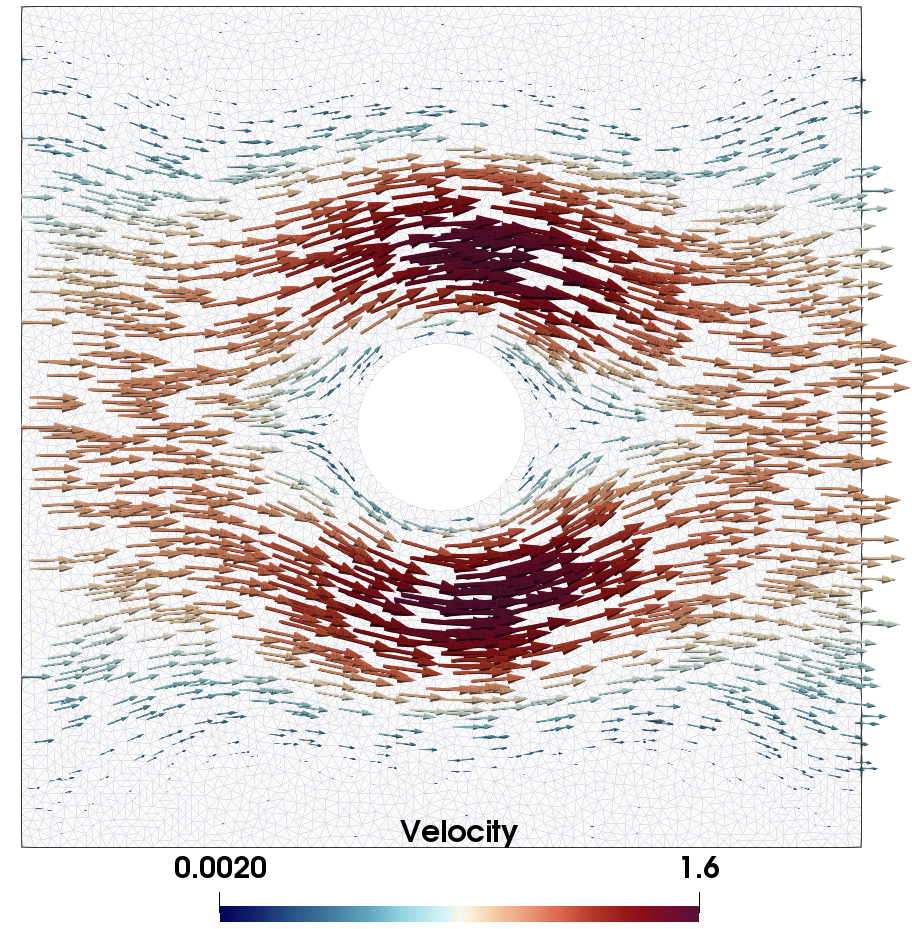}
	}
	\subfloat[Streamlines of $\bu$]{
		\includegraphics[width=0.33\textwidth]{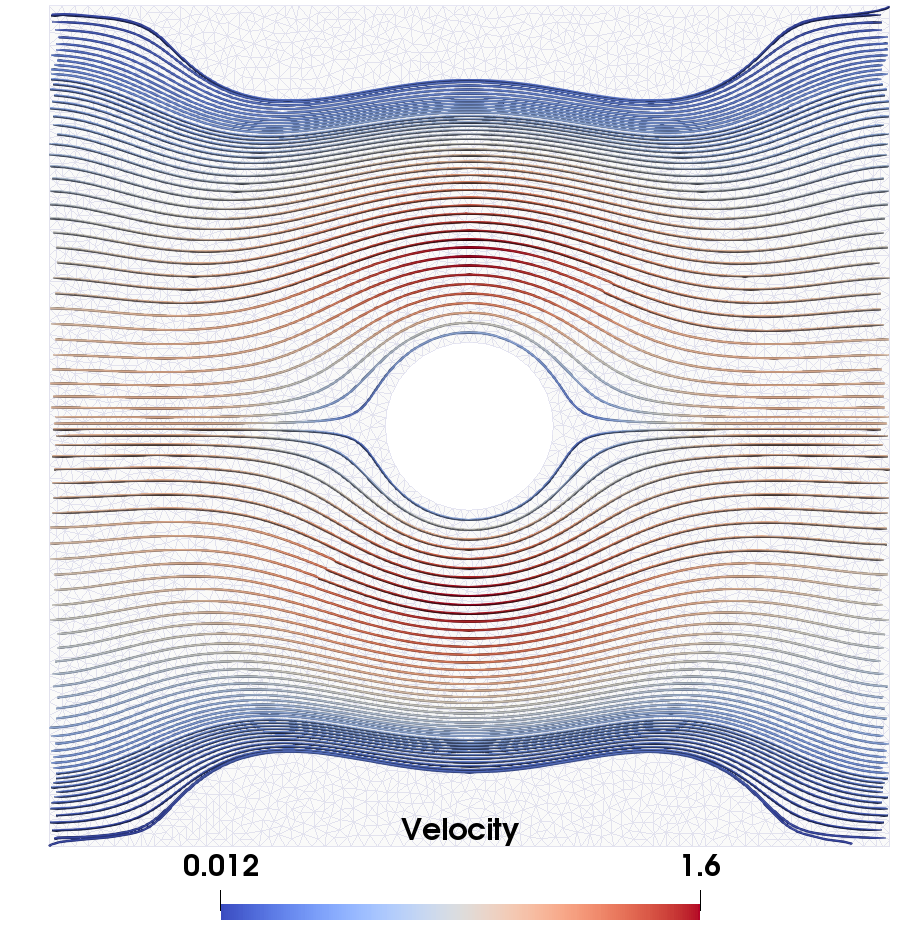}
	}
	\subfloat[Pressure p]{
		\includegraphics[width=0.33\textwidth]{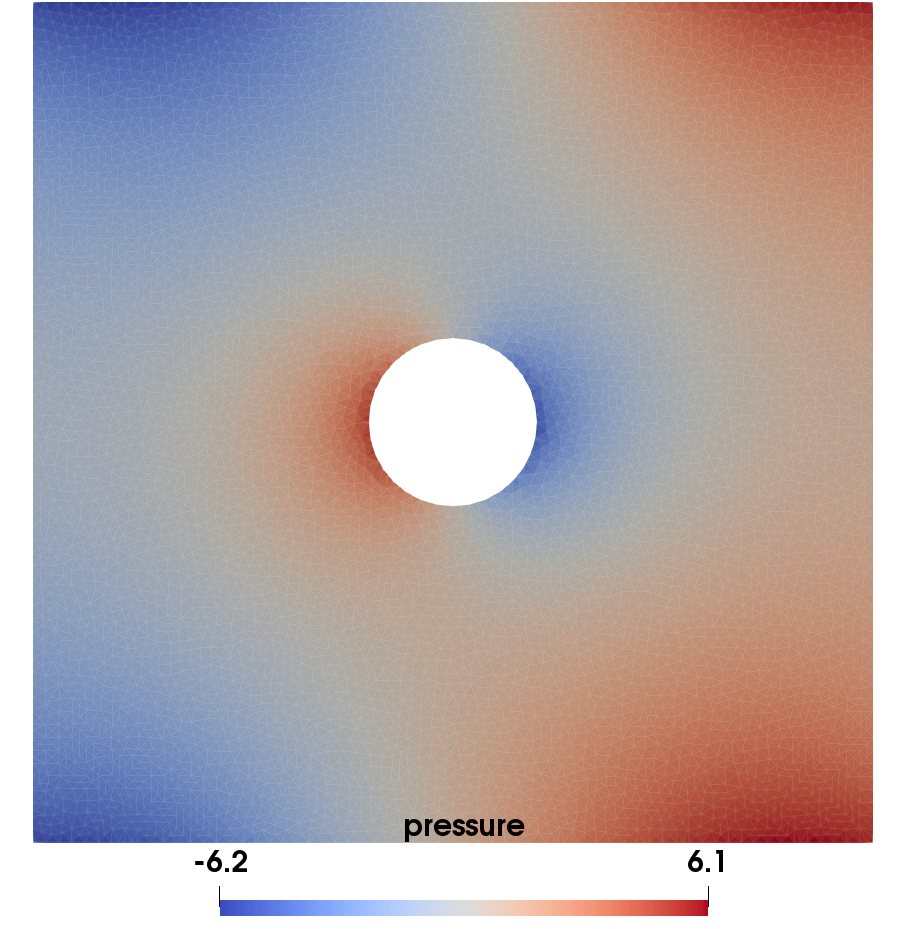}
	}
	\caption{Example 7.3: The velocity and pressure profiles for $\mu=0.1$.}
	\label{fig:ex4_10}
\end{figure}
\begin{figure}[!h]
	\centering
	\subfloat[Velocity $\bu$]{
		\includegraphics[width=0.33\textwidth]{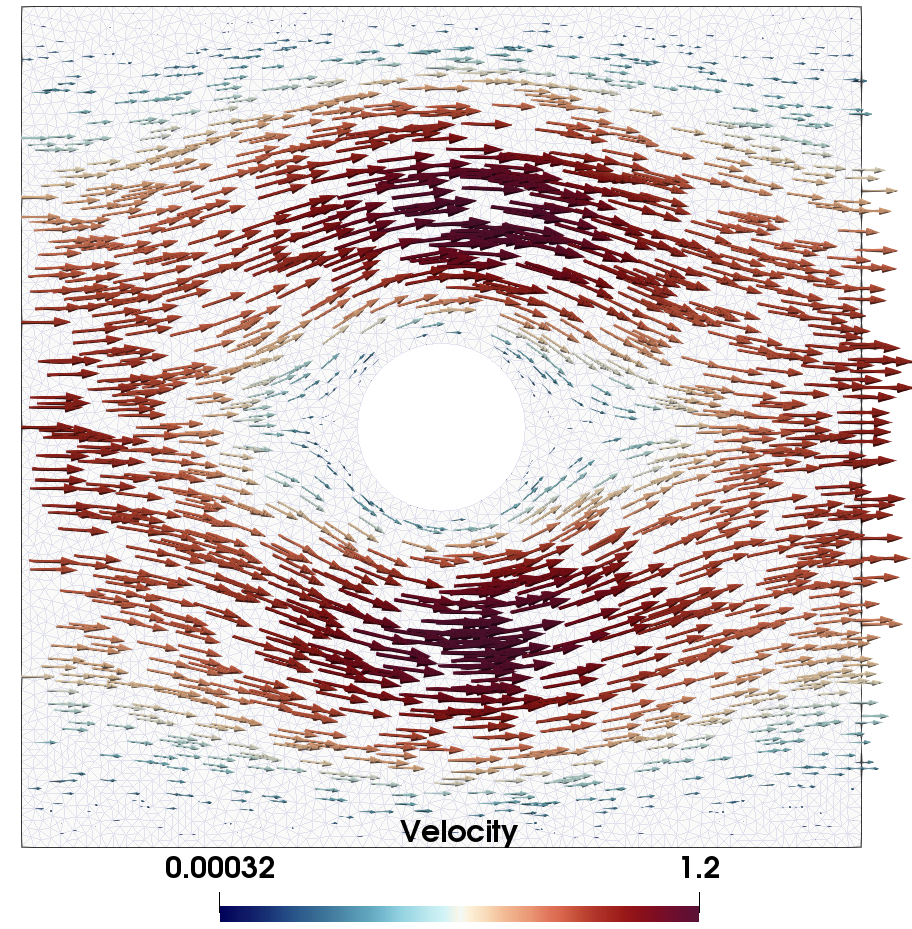}
	}
	\subfloat[Streamlines of $\bu$]{
		\includegraphics[width=0.33\textwidth]{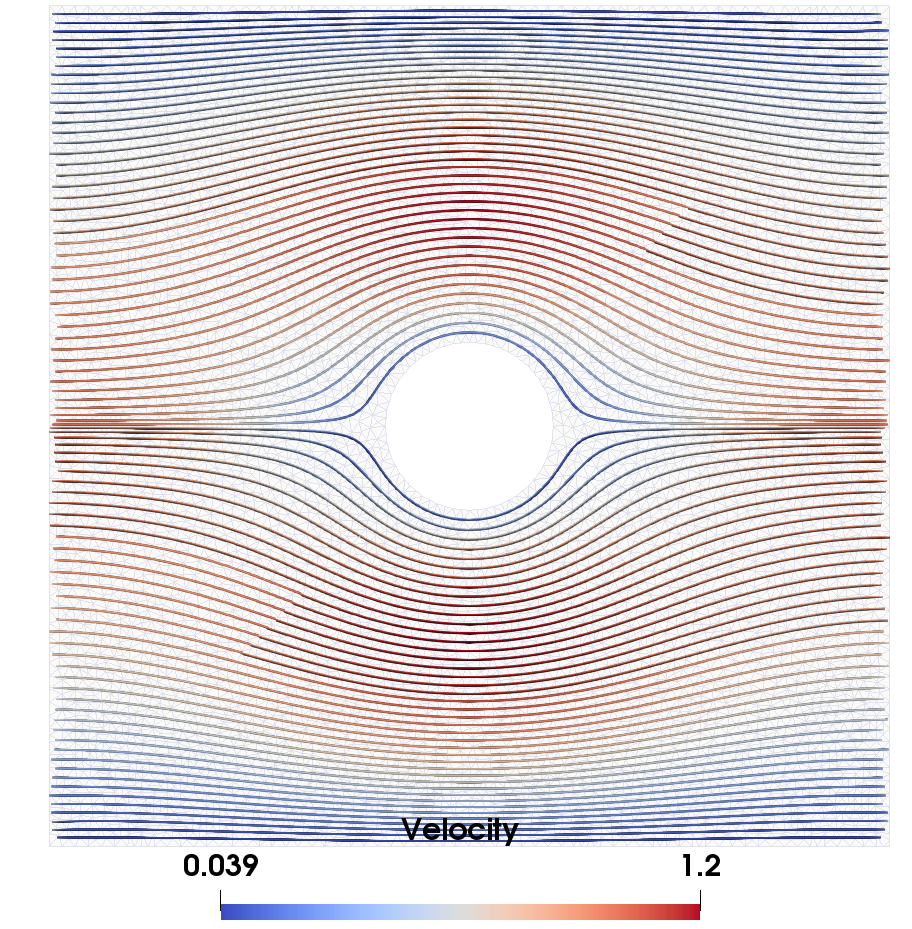}
	}
	\subfloat[Pressure p]{
		\includegraphics[width=0.33\textwidth]{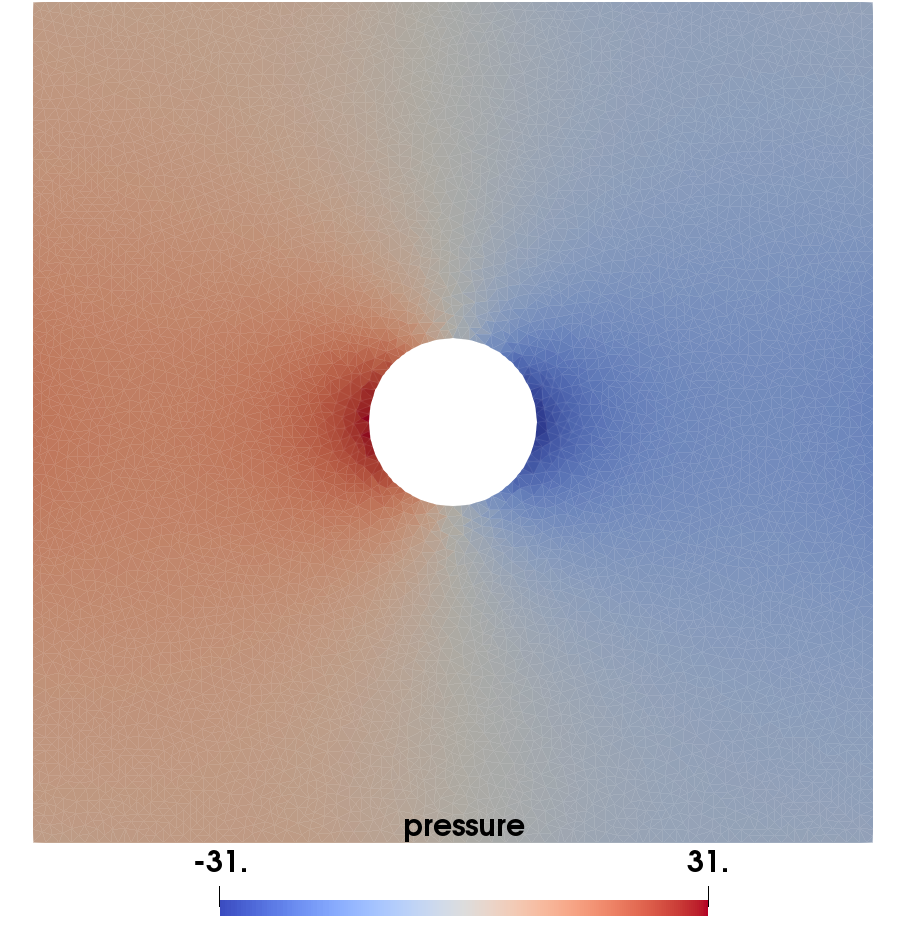}
	}
	\caption{Example 7.3: The velocity and pressure profiles for $\mu=1$.}
	\label{fig:ex4_1}
\end{figure}

\newpage 
Table~\ref{tab:prec-ex3} shows iterations counts for the block preconditioners. Here, we vary the physical parameter $\mu$ only while the tolerance is set to $10^{-6}$. Numerical results verify that the proposed block preconditioners are also effective for this more complicated test problem.

\begin{table}[!h]
	\begin{center}
		\begin{tabular}{|c || c| c| c ||c| c| c |}
			\hline
			$\mu$ & $\mathcal{B}_D$ & $\mathcal{B}_L$ & $\mathcal{B}_U$  & $\mathcal{M}_D$ & $\mathcal{M}_L$ & $\mathcal{M}_U$ 
			\\ 
			\hline 
			$1$ 	& 31 & 16 & 13 & 32 & 21 & 25  \\
			$0.1$ 	& 34 & 18 & 16 & 36 & 23 & 24 \\
			$0.01$ 	& 41 & 21 & 21 & 42 & 28 & 27  \\
			$0.001$ & 47 & 25 & 25 & 45 & 30 & 30  \\
			\hline
		\end{tabular}
	\end{center}
    \caption{Example~\ref{sec:ex3}. Iteration counts for the block preconditioners.}
	\label{tab:prec-ex3}
\end{table}%

\subsection{Example 4: Channel flow with Discontinuous Viscosity } \label{sec:ex4}
In this last example, we utilize the same domain, mesh, and boundary conditions, \eqref{eqn:ex3}, as in Example 3, but we set a discontinuous viscosity as follows:
\begin{equation*}
	\mu  = 
	\begin{cases}
			1 \  \ \ \ \  \text{ if } y > 0.5, \\
		0.01 \ \text{ if } y \leq 0.5.
	\end{cases}
	\label{eqn:viscosity}
\end{equation*}
Figure \ref{fig:ex4_discont} illustrates the vector field of the velocity, the streamlines of the velocity, and the pressure values.  Here, we see that EG provides numerical results without any spurious oscillations  near the discontinuity. Additionally, Table~\ref{tab:prec-ex4} summarizes the iteration counts for the block preconditioners with the discontinuous $\mu$. The proposed block preconditioners still perform effectively in this case with a discontinuous viscosity. 

\begin{figure}[!h]
	\centering
	\subfloat[Velocity $\bu$]{
		\includegraphics[width=0.33\textwidth]{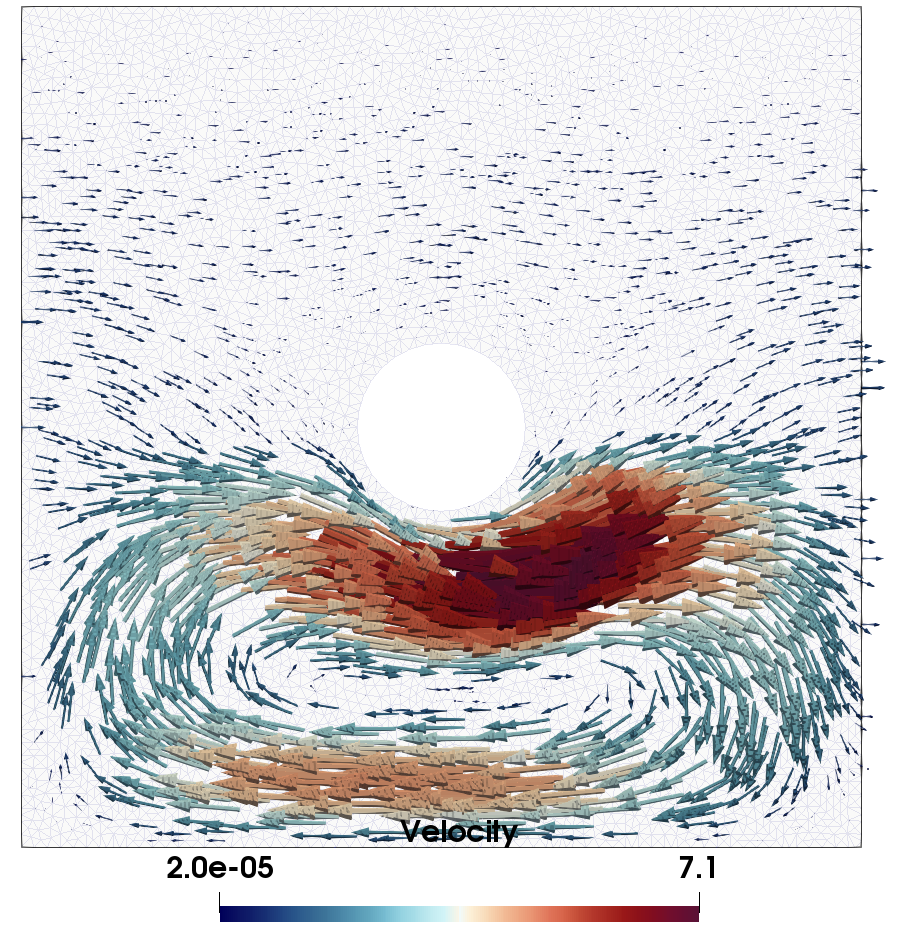}
	}
	\subfloat[Streamlines of $\bu$]{
		\includegraphics[width=0.33\textwidth]{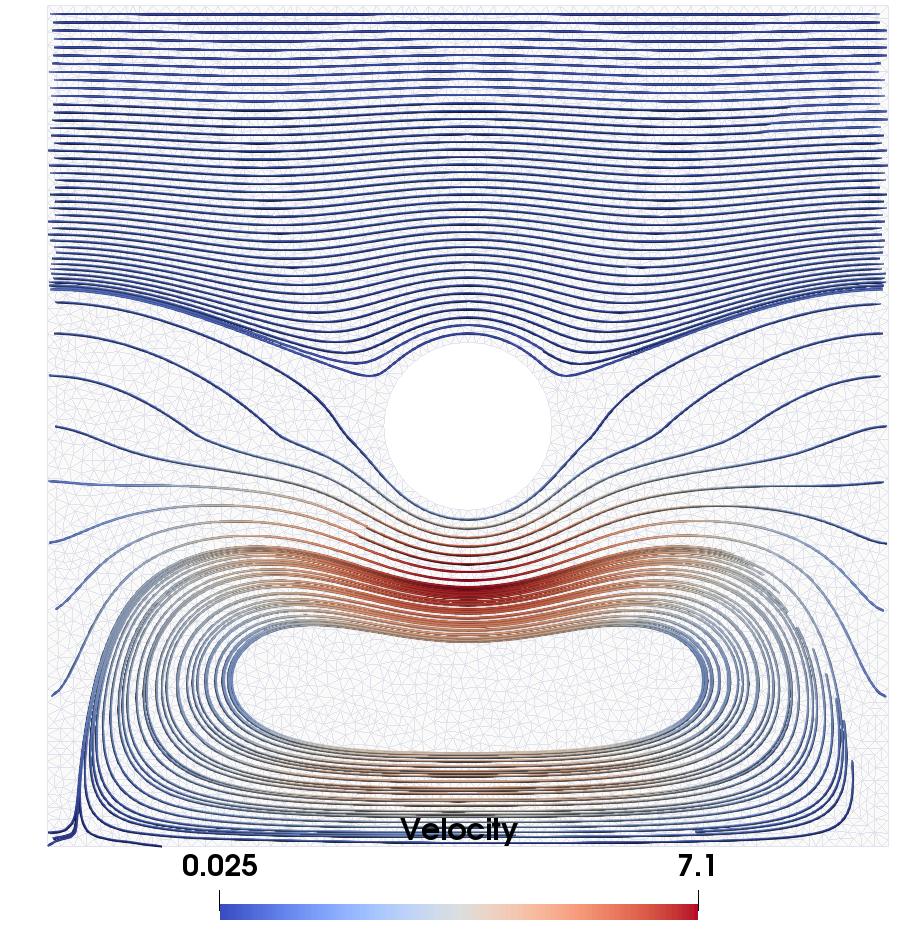}
	}
	\subfloat[Pressure p]{
		\includegraphics[width=0.33\textwidth]{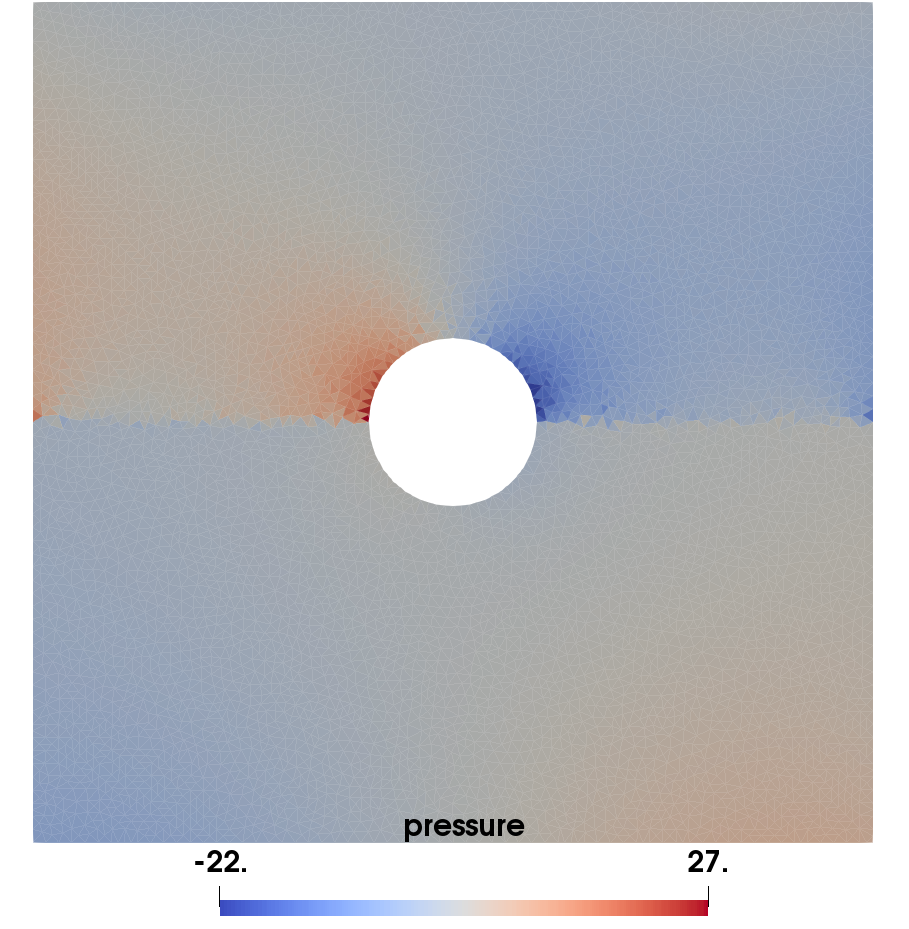}
	}
	\caption{Example 7.4: The velocity and pressure values with a discontinuous viscosity.}
	\label{fig:ex4_discont}
\end{figure}

\begin{table}[h!]
	\begin{center}
		\begin{tabular}{|c | c | c || c | c|  c  |c |}
			\hline 
		    $\mathcal{B}_D$ & $\mathcal{B}_L$ & $\mathcal{B}_U$  & $\mathcal{M}_D$ & $\mathcal{M}_L$ & $\mathcal{M}_U$ 
			\\ 
			\hline 
			 137 & 73 & 66 & 136 & 93 & 88 \\
			\hline 
		\end{tabular}
	\end{center}
    \caption{Example~\ref{sec:ex4}. Iteration counts for the block preconditioners.}
	\label{tab:prec-ex4}
\end{table}%

\section{Conclusions}\label{sec:conclusions}
In this paper, we have presented a new enriched Galerkin scheme for the Stokes equations based on the piecewise linear space for the velocity and piecewise constants for the pressure.  As this ``enrichment" involves adding only one degree of freedom per element, there are far fewer total degrees of freedom in this EG method than there are for a standard conforming continuous Galerkin approach, such as the $\mathbb{P}_2$-$\mathbb{P}_1$ or $\mathbb{P}_2$-$\mathbb{P}_0$ schemes and other nonconforming discontinuous Galerkin discretizations.  Yet, we are able to prove the well-posedness of the new EG scheme and show, via an error analysis, an optimal order of convergence for the velocity in the energy norm and the pressure in the $L^2$-norm.  Moreover, due to the well-posedness of the discrete system, we have shown that robust block preconditioners can be developed, yielding scalable results independent of the discretization and physical parameters.  The numerical results confirm this robustness for a variety of test problems with varying types of boundary conditions, and with more complex fluid dynamic features such as varying viscosities that are both continuous and discontinuous. Future work involves expanding the results to three-dimensional geometries and validating block preconditioners for those test problems as well.  Additionally, this approach can also be applied to other complex fluid problems such as Navier-Stokes and magnetohydrodynamics.

\bibliography{Stokes,stokes_intro,lit}

\end{document}